\documentclass[11pt,a4paper]{article}
\usepackage{amsfonts}
\usepackage{amssymb}
\usepackage{amsthm}
\usepackage{amsmath}
\usepackage{graphicx}
\usepackage{empheq}
\usepackage{indentfirst}
\usepackage{cite}
\usepackage{mathrsfs}
\usepackage{graphics}
\usepackage{dsfont}

\newtheorem{theorem}{\textbf{Theorem}}[section]
\newtheorem{lemma}{\textbf{Lemma}}[section]
\newtheorem{proposition}{\textbf{Proposition}}[section]
\newtheorem{corollary}{\textbf{Corollary}}[section]
\newtheorem{remark}{\textbf{Remark}}[section]
\newtheorem{definition}{\textbf{Definition}}[section]

\allowdisplaybreaks[4]

\def\be{\begin{equation}}
\def\ee{\end{equation}}
\def\bea{\begin{eqnarray}}
\def\eea{\end{eqnarray}}
\def\bt{\begin{theorem}}
\def\et{\end{theorem}}
\def\bl{\begin{lemma}}
\def\el{\end{lemma}}
\def\br{\begin{remark}}
\def\er{\end{remark}}
\def\bp{\begin{proposition}}
\def\ep{\end{proposition}}
\def\bc{\begin{corollary}}
\def\ec{\end{corollary}}
\def\bd{\begin{definition}}
\def\ed{\end{definition}}

\def\non{\nonumber }
\def\ub{\mathbf{u}}
\def\btau{\boldsymbol{\tau}}
\def\BH{\mathbf{H}}

\begin{document}

\title{Global Weak Solutions to the Navier--Stokes--Darcy--Boussinesq System for Thermal Convection in Coupled Free and Porous Media Flows}

\author{
 {\sc Xiaoming Wang}
 \footnote{Department of Mathematics, SUSTech International Center for Mathematics and National Center for Applied Mathematics Shenzhen, Southern University of Science and Technology, Shenzhen 518055, China.
 	Email: \texttt{wxm.math@outlook.com}}\ \
 and \ {\sc Hao Wu}
 \footnote{School of Mathematical Sciences and Shanghai Key Laboratory for Contemporary Applied Mathematics, Fudan University, Shanghai 200433, China.
 	Email: \texttt{haowufd@fudan.edu.cn, haowufd@yahoo.com.}
 }
 }

\date{\today}
\maketitle

\begin{abstract}
\noindent
We study the Navier--Stokes--Darcy--Boussinesq system that models the thermal convection of a fluid overlying a saturated porous medium in a general decomposed domain. In both two and three spatial dimensions, we first prove the existence of global weak solutions to the initial boundary value problem subject to the Lions and Beavers--Joseph--Saffman--Jones interface conditions. The proof is based on a proper time-implicit discretization scheme combined with the Leray--Schauder principle and  compactness arguments. Next, we establish a weak-strong uniqueness result such that a weak solution coincides with a strong solution emanating from the same initial data as long as the latter exists.
\medskip

 \smallskip

\end{abstract}


\section{Introduction}
\setcounter{equation}{0}

The study of the coupling free flow
and porous media flow is of considerable interest and
has attracted a lot of attentions in recent years due to its wide applications in geosciences (e.g., karst aquifers, hyporheic flow, contaminant transport), health sciences (e.g., blood flow) and industrial processes, see \cite{DQ09} and the references therein. In this paper, we investigate a  Navier--Stokes--Darcy--Boussinesq system that models thermal convection in an incompressible viscous fluid overlying a saturated porous medium (see, e.g., \cite{MMW19}). The convection phenomenon under consideration is much more complicated than that in a single fluid (cf. \cite{FMT1987} for the free-flow and \cite{F86,LT99} for fluids in a porous medium), since more physical parameters will affect the heat transport process. Linear and nonlinear stability analysis, properties of bifurcation and dynamic transition for the coupled system, with the Navier--Stokes equations and Darcy's equation governing the free-flow and the porous regions, have been provided in \cite{MMW19,HWW20} under suitable interface conditions. Here, our aim is to perform a first-step analysis on the well-posedness of the associated initial boundary value problem, proving the existence of global weak solutions and their uniqueness property.

Assume that the fluid is confined in a bounded connected domain $\Omega \subset \mathbb{R}^d$ ($d=2,3$) with $C^{2,1}$ boundary $\partial \Omega$. The unit outer normal vector on $\partial \Omega$ is denoted by $\mathbf{n}=\mathbf{n}(x)$. The domain $\Omega$ is partitioned into two non-overlapping regions such that $\overline{\Omega}=\overline{\Omega}_f\cup \overline{\Omega}_m$ and $\Omega_f \cap \Omega_m= \emptyset$, where $\Omega_f$ and $\Omega_m$ represent the free-flow region and the porous matrix region, respectively. We denote by $\partial \Omega_f$ and $\partial \Omega_m$ the boundaries of the free-flow and the matrix part, with $\widehat{\mathbf{n}}_f$, $\widehat{\mathbf{n}}_m$ being the corresponding unit outer normals on them. Both $\partial \Omega_f$ and $\partial \Omega_m$ are assumed to be Lipschitz continuous. The interface between the two parts (i.e., $\partial \Omega_f\cap \partial \Omega_m$) is denoted by $\Gamma_{i}$. On the free-flow/matrix interface $\Gamma_{i}$,  $\mathbf{n}_{i}$ stands for the unit normal on $\Gamma_{i}$ pointing from the free-flow region to the matrix, and $\{\btau_j\}$ $(j=1,...,d-1)$ stands for a local orthonormal basis for the tangent plane to $\Gamma_{i}$. Let $\Gamma_f=\partial \Omega_f\backslash \Gamma_{i}$ and $\Gamma_m=\partial \Omega_m\backslash \Gamma_{i}$ with $\mathbf{n}_f, \mathbf{n}_m$ being the unit outer normals to
$\Gamma_{f}$ and $\Gamma_{m}$. We assume that $\Gamma_m$ and $\Gamma_{i}$ have positive measure (i.e., $|\Gamma_m|>0$, $|\Gamma_{i}|>0$) but allow $\Gamma_f=\emptyset$, namely, $\Omega_f$ can be enclosed completely by $\Omega_m$. When $d=3$, we assume that the surfaces $\Gamma_f$, $\Gamma_m$ and $\Gamma_{i}$ have Lipschitz continuous boundaries.

In the sequel, the subscript $f$  (or $m$) indicates that the variables are for the free-flow part (or the matrix part). We denote by $\ub$ the mean velocity of the (incompressible) fluid and $\theta$ the (relative) temperature of the fluid. The following convention will be assumed throughout the paper
$$
\ub|_{\Omega_f}=\ub_f, \ \ \ \ \ub|_{\Omega_m}=\ub_m, \ \ \ \ \theta|_{\Omega_f}=\theta_f, \ \ \theta|_{\Omega_m}=\theta_m.
$$

\textbf{Governing PDE system}. We shall consider the following Navier--Stokes--Darcy--Boussinesq system (in a nondimensionallized form, see \cite{MMW19})
\begin{align}
&\partial_t \ub_f+(\ub_f\cdot\nabla )\ub_f=\nabla\cdot\mathbb{T}(\ub_f,P_f) + \theta_f\mathbf{k}, &\mbox{in}\ \Omega_f\times(0,T),
\label{uf1}\\
&\nabla\cdot\ub_f=0, &\mbox{in}\ \Omega_f\times(0,T),
\label{uf2}\\
&\partial_t\theta_f+(\ub_f\cdot\nabla)\theta_f={\rm div}(\lambda_f(\theta_f)\nabla\theta_f), &\mbox{in}\ \Omega_f\times(0,T),
\label{tf1}\\
&\varpi\partial_t\ub_m+\nu(\theta_m)\mathbb{K}^{-1}\ub_m
=-\nabla P_m + \theta_m\mathbf{k},
&\mbox{in}\ \Omega_m\times(0,T),
\label{um1}\\
&\nabla\cdot\ub_m=0, &\mbox{in}\ \Omega_m\times(0,T), \label{um2}\\
&\partial_t\theta_m+(\ub_m\cdot\nabla)\theta_m={\rm div}(\lambda_m(\theta_m)\nabla\theta_m), &\mbox{in}\ \Omega_m\times(0,T),
\label{tm1}
\end{align}
where $T\in (0,+\infty)$ is an arbitrary final time and $\mathbf{k}$ is the upward pointing unit vector. In the free-flow region, the motion of the incompressible fluid is characterized by the Navier--Stokes equations \eqref{uf1}--\eqref{uf2} with Boussinesq approximation (buoyancy force), coupled with the advection-diffusion equation \eqref{tf1} for the temperature. While for the fluid in porous medium, we employ the Darcy system \eqref{um1}--\eqref{um2} (valid under the small porosity assumption that is generally applicable to geophysical systems) with the advection-diffusion equation \eqref{tm1}. The Cauchy stress tensor $\mathbb{T}$ in equation  \eqref{uf1}  is given by
\begin{align}
\mathbb{T}(\ub_f,p_f) = 2\nu(\theta_f)\mathbb{D}(\ub_f)-P_f\mathbb{I},
\label{CT}
\end{align}
where $\mathbb{D}(\ub_f)=\frac{1}{2}(\nabla \ub_f + \nabla^T \ub_f)$ is the symmetric rate of deformation tensor and $\mathbb{I}$ denotes the $d\times d$ identity matrix. The scalar functions $P_f$ and $P_m$ stand for the pressures in the free-flow and matrix regions. The fluid viscosity is denoted by $\nu$. The thermal diffusivity coefficients may differ in the free-flow and matrix regions and are denoted by $\lambda_f$, $\lambda_m$, respectively. The viscosity and thermal diffusivity are allowed to be functions that may depend on the temperature $\theta$, which are physically important in the study of non-isothermal fluids (see, e.g., \cite{LB96}). In equation \eqref{um1}, $\mathbb{K}$ is a $d\times d$ matrix standing for the permeability of the porous medium,  which is usually assumed to be a bounded, symmetric and uniformly positive definite matrix but could be heterogeneous \cite{Bear}. The parameter $\varpi$ in \eqref{um1} is a nonnegative constant related to the so-called Darcy--Prandtl number.
Since the Darcy--Prandtl number for porous medium flows is heuristically small in the regime with a small Darcy number \cite{Jo76}, the term  $\varpi\partial_t\ub_m$ is often neglected in the literature (see e.g., \cite{LT99,Va}), while in some other works, this time derivative term is kept primarily for the benefit of the energy analysis, and it would allow more accurate description of temporal transitions \cite{MMW19}. In this paper, we shall treat both cases with or without this time derivative term. For the sake of simplicity, physical coefficients that are not important for our subsequent mathematical analysis are set to be one (for instance, the fluid density and those in the Boussinesq approximation related to the gravitational acceleration and thermal expansion coefficient etc).

Next, we describe the initial, boundary as well as interface conditions of the coupled system \eqref{uf1}--\eqref{CT}.

\textbf{Initial conditions}. The system \eqref{uf1}--\eqref{CT} is subject to the initial conditions
\begin{align}
& \ub_f|_{t=0}=\ub_{0f}(x), \qquad \text{in}\ \Omega_f,\quad \text{and}\qquad \ub_m|_{t=0}=\ub_{0m}(x),  &\text{in}\  \Omega_m,\label{ini1}\\
& \theta|_{t=0}=\theta_{0}(x), &\text{in}\ \Omega. \label{ini2}
\end{align}
 In particular, when $\varpi=0$, we do not need to specify the initial velocity $\ub_{0m}$ as it can be recovered from $\theta_{0m}$ by solving the Darcy equation (see e.g., \cite{LT99}). \medskip

\textbf{Boundary conditions on $\Gamma_f$ and $\Gamma_m$}. Since we are mainly interested in the coupling on the  interface $\Gamma_i$ between sub-domains, we impose the following standard boundary conditions on the outer boundaries $\Gamma_f$ and $\Gamma_m$:
\begin{align}
& \ub_f=\mathbf{0}, &\text{on}\ \Gamma_f\times(0,T),
\label{IBC0}\\
& \ub_m\cdot\mathbf{n}_m=0, &\text{on}\ \Gamma_m\times(0,T),
\label{IBC1}\\
& \theta_f=0, &\text{on}\ \Gamma_f\times(0,T),
\label{IBC2}\\
& \theta_m=0, &\text{on}\ \Gamma_m\times(0,T).
\label{IBC3}
\end{align}

\textbf{Interface conditions on $\Gamma_{i}$}. Now on $\Gamma_i$, we assume that the system \eqref{uf1}--\eqref{CT} are coupled through the following set of interface conditions:
\begin{align}
&\ub_f\cdot\mathbf{n}_{i} = \ub_m\cdot\mathbf{n}_{i},
&\mbox{on}\ \Gamma_{i}\times(0,T),
\label{IBCi5}\\
&-\mathbf{n}_{i}\cdot(\mathbb{T}(\ub_f,P_f){\mathbf{n}_{i}})+\frac12|\ub_f|^2=P_m,
&\mbox{on}\ \Gamma_{i}\times(0,T),
\label{IBCi6}\\
&-\btau_j\cdot(\mathbb{T}(\ub_f,P_f){\mathbf{n}_{i}})
=\frac{\alpha\nu(\theta_m)}{\sqrt{\text{trace}(\mathbb{K})}}\btau_j\cdot\ub_f,\quad j=1,..,d-1,
&\mbox{on}\ \Gamma_{i}\times(0,T),
\label{IBCi7}\\
&\theta_f=\theta_m,&\mbox{on}\ \Gamma_{i}\times(0,T),
\label{IBCi1}\\
& \lambda_f(\theta_f)\frac{\partial \theta_f}{\partial \mathbf{n}_{i}}=\lambda_m(\theta_m)\frac{\partial \theta_m}{\partial \mathbf{n}_{i}},
&\mbox{on}\ \Gamma_{i}\times(0,T).
\label{IBCi2}
\end{align}
The condition \eqref{IBCi5} indicates the continuity in normal velocity that guarantees the conservation of mass, i.e., the exchange of fluid
between the two sub-domains is conservative. The condition \eqref{IBCi6}
represents the balance of the forces normal to the interface taking into consideration of the so-called dynamic pressure $\frac12|\ub_f|^2$. With this quadratic term, condition \eqref{IBCi6} is known as the Lions interface condition in the literature (see e.g., \cite{CR09,DQ09,GR09}). This specific choice gives rise to a dissipative energy law that facilitates the analysis significantly \cite{CR08}. On the other hand, in the absence of this dynamic pressure term, the kinetic energy could increase without external forcing, which is physically unrealistic. Recently in \cite{MMW19}, the authors showed that the nonlinear dynamic pressure term is somewhat small, thus the difference between solutions produced with the Lions interface condition and its linear counterpart is heuristically small as well. More precisely, with a formal asymptotic argument, they showed that the order of the dynamic pressure term is $\mathcal{O}(\mathrm{Da})$ in the small Darcy number limit as $\mathrm{Da}\to 0$ and it begins to affect solutions to the perturbed systems at the scale
$\mathcal{O}(\mathrm{Da}^2)$ (see \cite[Appendix]{MMW19}). 

Next, the interface condition \eqref{IBCi7} is
the so-called Beavers--Joseph--Saffman--Jones (BJSJ) condition (cf. \cite{Jo73, Sa71}), where $\alpha> 0$ is an empirical constant usually determined by the domain geometry and the porous material. We note that the BJSJ condition is a simplified variant of the well-known Beavers--Joseph (BJ) condition (see \cite{B-J}) that addresses how the porous medium affects the conduit flow at the interface:
$$-\btau_j\cdot(2\nu\mathbb
D(\mathbf{u}_f))\mathbf{n}_{i}
=\frac{\alpha_{BJ}\nu}{\sqrt{\text{trace}(\mathbb{K})}}\btau_j\cdot(\ub_f-\ub_m), \ \ \mbox{on}\ \Gamma_{i}, \ j=1,...,d-1.
$$
 This empirical condition essentially says that the
tangential component of the normal stress that the free-flow incurs along the interface is proportional to the jump in the tangential velocity over the interface. To get the BJSJ condition, the term $-\btau_j\cdot \ub_m$ on the right-hand
side is simply dropped from the corresponding BJ condition
(as long as the Darcy number is small). Mathematically rigorous justification of this simplification under appropriate assumptions can be found in \cite{J-M}. At last, the interface conditions \eqref{IBCi1}--\eqref{IBCi2} involving $\theta$ are simply the continuity conditions for the temperature functions and their conormal derivatives across the interface (sometimes also referred to as transmission conditions).

The coupling system \eqref{uf1}--\eqref{CT} contains the Navier--Stokes--Darcy system for an incompressible viscous flow as a subsystem, which has been extensively studied in the literature. We just mention a few results related to the  mathematics analysis and refer readers to the references cited therein. 

For the simplified case of coupled (stationary) Stokes and Darcy equations, analysis of weak solutions has been done in \cite{CGHW10,LSY03,DMQ01}. We also refer to \cite{HWW} for the existence of global weak solutions and weak-strong uniqueness of a Cahn--Hilliard--Stokes--Darcy system for incompressible two-phase fluids. For the coupling of Navier--Stokes and Darcy equations, the stationary case has been studied in \cite{BDQ10,ChR09,DQ09, GR09}, and in \cite{CR12}, a Navier--Stokes/Darcy--transport system modelling the contamination of ground water was analyzed. On the other hand, the time-dependent problem have been investigated in \cite{CR08,CR09,CVR13} with different choices of interface conditions. In those works, existence and uniqueness of weak solutions are obtained under certain specific conditions, for instance, the small data assumption. 

We also note that our system \eqref{uf1}--\eqref{CT} contains the Boussinesq system either for the free-flow \cite{FMT1987,HL05,LPZ11,LT16,LB96,LB99,SZ13,Temam1997,WZ11} or for the flow in a porous medium \cite{F86,LT99}, which have been studied in a vast literature under various settings. In our current case, we have to deal with new difficulties due to the complicated coupling of flows governed by different physical processes, the complex geometry of domain, and in particular, the coupled nonlinear interface conditions.

The aim of this paper is two-fold. First, we prove the existence of global weak solutions to problem \eqref{uf1}--\eqref{IBCi2} with variable fluid viscosity, thermal diffusivity and a nonnegative Darcy--Prandtl number $\varpi$ in both two and three dimensions (see Theorem \ref{thmEx}). Here we choose to work with the Lions and BJSJ interface conditions, which lead to a dissipative energy law for the solutions so that no smallness assumptions on the initial data are necessary (cf. \cite{CR08,CVR13}). The proof is based on a semi-implicit discretization scheme with a Brinkman type regularization that can be solved by the Leray--Schauder principle (cf. \cite{ADG13,HWW} for related applications to some hydrodynamic systems for two-phase flows). Using the energy inequality, we derive uniform estimates of approximate solutions and then obtain the existence of global weak solutions to the original system by a compactness argument. 

Second, we prove the weak-strong uniqueness property of problem \eqref{uf1}--\eqref{IBCi2} (see Theorem \ref{thmuniq}), namely, a weak solution coincides with a strong solution emanating from the same initial data as long as the latter exists.  Uniqueness of weak solutions to problem \eqref{uf1}--\eqref{IBCi2} is not trivial even when the spatial dimension is two due to the nonlinear interface condition of Lions type. Besides the Navier--Stokes equations, additional difficulties come from the thermal advection terms, low regularity of the solution $\ub_m$ to Darcy's equation and variable viscosity/thermal diffusivity coefficients. We remark that our approach can be further applied to study the Cahn--Hilliard--Navier--Stokes--Darcy--Boussinesq system for  thermal convection of two-phase flows in a fluid layer overlying a porous medium (see e.g., \cite{CHWZ20}).

Finally, due to the decomposed domain setting in the B\'enard convection problem \eqref{uf1}--\eqref{IBCi2} and in particular, the complicated coupling via interface conditions, existence of strong or classical solutions (especially with higher-order spatial regularity) remains a challenging open problem, which is quite different from the case in a simple domain (cf. \cite{FMT1987,HL05,LPZ11,LB99,LT99,SZ13} and the references therein). We refer to \cite{LL19} for an attempt on the  Navier--Stokes--Darcy system for isothermal incompressible flows in a two-dimensional strip domain.

The rest of this paper is organized as follows.
In Section 2, we introduce the function spaces,  present the definition of weak solutions and state the main results. Section 3 is devoted to the proof for the existence of global weak solutions to problem \eqref{uf1}--\eqref{IBCi2}. In Section 4, we prove the weak-strong uniqueness property by the energy method.


\section{Main Results}
\setcounter{equation}{0}

\subsection{Preliminaries}
Let $d=2,3$ be the spatial dimension. For arbitrary vectors $\mathbf{a},\mathbf{b}\in \mathbb{R}^d$, we denote $\mathbf{a}\otimes \mathbf{b}=(a_jb_k)^d_{j,k=1}$ and $\mathbf{a}\cdot \mathbf{b}=\sum_{j=1}^d a_jb_j$.
Let $X$ be a Banach space with its dual denoted by $X'$. We denote by
$\langle u, v\rangle \equiv \langle u, v\rangle_{X'
,X} $ the duality product for $u \in X'$, $v\in X$. The inner product on a Hilbert space $H$ is denoted by $(\cdot, \cdot)_H$.
We use $L^q(\Omega)$, $1 \leq q \leq +\infty$ to denote the usual Lebesgue space on $\Omega$ and $\|\cdot\|_{L^q(\Omega)}$ for its norm. Similarly, $W^{m,q}(\Omega)$, $m \in \mathbb{N}$, $1 \leq q \leq + \infty$, denote the usual
Sobolev spaces with norm $\|\cdot \|_{W^{m,p}(\Omega)}$, and for $q=2$, we simply denote $W^{m,2}(\Omega)$ by $H^m(\Omega)$. The fractional order Sobolev spaces $H^s(\Omega)$ ($s\in \mathbb{R}$) are defined as in \cite[Section 4.2.1]{Tri}.  If $I$ is an interval of $\mathbb{R}^+$, we use $L^p(I;X)$, $1 \leq p \leq +\infty$, to denote the function space that consists of $p$-integrable
functions with values in $X$. Moreover, $C_w(I;X)$ denotes the topological space of all bounded and weakly continuous functions from $I$ to $X$, while $W^{1,p}(I;X)$ with $1\leq q<+\infty$ stands for the space of all functions $u$ such that $u, \frac{du}{dt}\in L^p(I;X)$, where $\frac{du}{dt}$ denotes the distributional time derivative of $u$. Bold characters are used to denote vector or matrix valued spaces.

Let $\Omega$ be the domain with decomposition that has been described in the Introduction. For our problem, we introduce the following spaces
\bea
 \BH({\rm div}; \Omega_k) &:=&\{\mathbf{w}\in \mathbf{L}^2(\Omega_k)~|~\nabla \cdot \mathbf{w}\in L^2(\Omega_k)\},\non\\
\widetilde{\mathbf{H}}_{k,0}&:=&\{\mathbf{w}\in \BH({\rm div}; \Omega_k)~|~\mathbf{w}\cdot \mathbf{n}_k=0 \ \text{on}\ \Gamma_{k}\}, \non\\
\widetilde{\mathbf{H}}_{k,\mathrm{div}}
&:=&\{\mathbf{w}\in\widetilde{\mathbf{H}}_{k,0}~|~\nabla \cdot\mathbf{w}=0\},\non\\
\mathbf{H}_{k,0}&:=&\{\mathbf{w}\in \BH^1(\Omega_k)~|~\mathbf{w}=\mathbf{0}\text{ on
}\Gamma_{k}\}, \non\\
\mathbf{H}_{k,\text{div}}&:=&\{
\mathbf{w}\in\mathbf{H}_{k,0}~|~\nabla \cdot\mathbf{w}=0\},\nonumber\\
\widehat{\mathbf{H}}_{k,0}&:=&\{\mathbf{w}\in \BH^1(\Omega_k)~|~\mathbf{w}=\mathbf{0}\text{ on
}\partial\Omega_{k}\}, \non\\
\widehat{\mathbf{H}}_{k,\text{div}}&:=&\{
\mathbf{w}\in\widehat{\mathbf{H}}_{k,0}~|~\nabla \cdot\mathbf{w}=0\},\nonumber
 \eea
 with index $k\in \{f,m\}$.
For simplicity, we denote $(\cdot, \cdot)_f$, $(\cdot, \cdot)_m$ the inner products on the spaces $L^2(\Omega_f)$, $L^2(\Omega_m)$, respectively (also for the corresponding vector or matrix valued spaces). The inner product on $L^2(\Omega)$ is simply denoted by $(\cdot, \cdot)$. For any function $u\in L^2(\Omega)$ with $u_m:=u|_{\Omega_m}$ and $u_f:=u|_{\Omega_f}$, it holds
$$
(u,v)=(u_m,v_m)_m+(u_f,v_f)_f, \quad \|u\|_{L^2(\Omega)}^2=\|u_m\|_{L^2(\Omega_m)}^2+\|u_f\|_{L^2(\Omega_f)}^2.
$$
On the interface $\Gamma_{i}$, we consider the fractional Sobolev spaces
$H^{\frac12}_{00}(\Gamma_{i})$ and $H^\frac12(\Gamma_{i})$ for a (Lipschitz) surface $\Gamma_{i}$ when $d=3$ or a curve when $d=2$, with the following equivalent norms (see \cite[Chapter 1, Section 11]{L-M}, or \cite{Gri85}):
\begin{eqnarray}
&& \|u\|_{H^\frac12(\Gamma_{i})}^2=\int_{\Gamma_{i}} |u|^2 dS +\int_{\Gamma_{i}}\!\int_{\Gamma_{i}}\frac{|u(x)-u(y)|^2}{|x-y|^d} dxdy,\nonumber\\
&& \|u\|_{H_{00}^\frac12(\Gamma_{i})}^2=\|u\|_{H^\frac12(\Gamma_{i})}^2
+\int_{\Gamma_{i}} \frac{|u(x)|^2}{\rho(x, \partial \Gamma_{i})} dx,\nonumber
\end{eqnarray}
where $\rho(x, \partial \Gamma_{i})$ denotes the distance from $x$ to $\partial \Gamma_{i}$. We note that these norms are not equivalent except when $\Gamma_{i}$  is a closed surface or curve. Besides, if $\Gamma_{i}$ is a subset of $\partial \Omega_f$ with positive measure, then $H^{\frac12}_{00}(\Gamma_{i})$ is a trace space of functions of $H^1(\Omega_f)$ that vanish on $\Gamma_{f}$ (see \cite{CVR13}). Similarly  in the vectorial case, we have $\BH^{\frac12}_{00}(\Gamma_{i})=\BH_{f,0}|_{\Gamma_{i}}$.  $H^{\frac12}_{00}(\Gamma_{i})$ is a non-closed subspace of
$H^{\frac12}(\Gamma_{i})$ and it has a continuous zero extension to
$H^{\frac12}(\partial\Omega_{f})$.
Moreover, we have the following continuous embedding result (see \cite{CGHW10}): $H^{\frac12}_{00}(\Gamma_{i})\subsetneqq H^{\frac12}(\Gamma_{i}) \subsetneqq H^{-\frac12}(\Gamma_{i}) \subsetneqq(H^{\frac12}_{00}(\Gamma_{i}))'$.
Let $H^{-\frac12}(\partial\Omega_f)|_{\Gamma_{i}}$ be defined in the
following way: for all $f\in H^{-\frac12}(\partial\Omega_f)|_{\Gamma_{i}}$ and $ g\in
H^{\frac12}_{00}(\Gamma_{i})$,
$\langle f,g\rangle_{H^{-\frac12}(\partial\Omega_f)|_{\Gamma_{i}},\,
H^{\frac12}_{00}(\Gamma_{i})}:=\langle f,\widetilde g\rangle_{H^{-\frac12}(\partial\Omega_f),\,
H^{\frac12}(\partial\Omega_f)}$ with $\widetilde g$ being the zero
extension of $g$ to $\partial\Omega_f$. Then we note that
$H^{-\frac12}(\partial\Omega_f)|_{\Gamma_{i}} \subset(H^{\frac12}_{00}(\Gamma_{i}))'$ but
$H^{-\frac12}(\partial\Omega_f)|_{\Gamma_{i}} \nsubseteq H^{-\frac12}(\Gamma_{i})$.
For any vector $\mathbf{u} \in \mathbf{H}({\rm div}; \Omega_f)$, its normal component  $\mathbf{u} \cdot \mathbf{n}_{i}$ is well defined in $(H^{\frac12}_{00}(\Gamma_{i}))^\prime$,  and for all $q \in H^1(\Omega_f)$ such that $q =0$ on $\Gamma_{f}$, we have
\begin{align}
(\nabla \cdot \mathbf{u}, q)_f=-(\mathbf{u}, \nabla q)_f+\langle \mathbf{u}\cdot \mathbf{n}_{i}, q\rangle_{(H^{\frac12}_{00}(\Gamma_{i}))^\prime,\,  H^{\frac12}_{00}(\Gamma_{i})}.\non
\end{align}
Similar results hold also on the sub-domain $\Omega_m$.

In our decomposed domain setting, the boundary $\Gamma_f=\emptyset$ is allowed, i.e., $\Omega_f$ can be enclosed completely by the matrix part $\Omega_m$. Since the classical Korn's inequality (see, e.g., \cite{Ho}) may not apply when $\Gamma_f=\emptyset$, in order to overcome this difficulty, we introduce the space
\begin{align}
 \mathbf{Z}=\big\{\mathbf{u}\ |\ \ub_f=\ub|_{\Omega_f}\in \mathbf{H}_{f, \mathrm{div}},\ \ub_m=\ub|_{\Omega_m}\in \widetilde{\mathbf{H}}_{m,\mathrm{div}}, \ \ub_f\cdot\mathbf{n}_{i}=\ub_m\cdot\mathbf{n}_{i} \ \text{on}\ \Gamma_{i}\big\},
\label{Z}
\end{align}
whose natural norm can be given by $\|\mathbf{u}_f\|_{\mathbf{H}^1(\Omega_f)}+
\|\mathbf{u}_m\|_{\mathbf{L}^2(\Omega_m)}$.
In view of \cite[Lemma 3.9]{HWW}, we have  the following result:
\begin{lemma}
\label{equinorml}
The norm defined by
\begin{align}\label{equinorm}
\|\mathbf{u}\|_{\mathbf{Z}}^2
:=\|\mathbb{D}(\mathbf{u}_f)\|_{\mathbf{L}^2(\Omega_f)}^2
+\sum_{j=1}^{d-1}\|\mathbf{u}_f \cdot \mathbf{\btau}_j\|_{L^2(\Gamma_{i})}^2+
\|\mathbf{u}_m\|_{\mathbf{L}^2(\Omega_m)}^2
\end{align}
 is an equivalent norm on $\mathbf{Z}$. There exists a constant $C$ independent of $\mathbf{u}$ such that
 $$
 \|\mathbf{u}_f\|_{\mathbf{H}^1(\Omega_f)}^2+
\|\mathbf{u}_m\|_{\mathbf{L}^2(\Omega_m)}^2\leq C\|\mathbf{u}\|_{\mathbf{Z}}^2,\qquad \forall\, \mathbf{u}\in \mathbf{Z}.
 $$
\end{lemma}


\subsection{Main results}\label{WM}
We make the following assumptions on the viscosity $\nu$,   thermal diffusivity $\lambda_f$, $\lambda_m$ as well as the permeability matrix $\mathbb{K}$.
 \begin{itemize}
 \item[(A1)] $\nu\in C^1(\mathbb{R})$,  $\underline{\nu} \leq \nu(s)\leq\bar{\nu}  $ and $|\nu'(s)|\leq \tilde{\nu}$ for $s\in \mathbb{R}$, where $\bar{\nu}$, $\underline{\nu}$ and $\tilde{\nu}$ are positive constants.
 \item[(A2)] ${\lambda}_j\in C^1  (\mathbb{R})$, $\underline{\lambda} \leq {\lambda}_j(s)\leq \bar \lambda $ and $|{\lambda}_j'(s)|\leq \tilde{\lambda}$ for $s\in \mathbb{R}$, where $\bar{\lambda}$, $\underline{\lambda}$ and $\tilde{\lambda}$  are positive constants, $j\in \{f,m\}$.
 \item[(A3)]  The permeability matrix $\mathbb{K}$ is isotropic, bounded from above and below,  namely, $\mathbb{K}=\kappa(x)\mathbb{I}$ with $\mathbb{I}$ being the $d\times d$ identity matrix and $\kappa(x)\in L^\infty (\Omega)$ such that there exist  $\bar{\kappa}>\underline{\kappa}>0$, $\underline{\kappa}\leq \kappa(x)\leq \bar{\kappa}$ a.e. in $\Omega$.
 \end{itemize}

Next, we introduce the notion of weak solutions.

\begin{definition}\label{defweak}
Suppose that $d=2,3$ and $T>0$ is arbitrary.

\textbf{Case 1}: $\varpi>0$. Consider the initial data $\ub_{0f}(x)\in \widetilde{\mathbf{H}}_{f,\mathrm{div}}$, $\ub_{0m}(x)\in \widetilde{\mathbf{H}}_{m,\mathrm{div}}$ with
$\ub_{0f}\cdot\mathbf{n}_{i} = \ub_{0m}\cdot\mathbf{n}_{i}$
on $\Gamma_{i}$, and $\theta_{0}\in L^2(\Omega)$. The triple $(\ub_f, \mathbf{u}_m, \theta)$ satisfying
\bea
&& \ub_f\in C_w([0,T]; \widetilde{\mathbf{H}}_{f,\mathrm{div}})\cap L^2(0, T; \mathbf{H}_{f,\mathrm{div}})\cap W^{1,\frac43}(0,T; (\mathbf{H}_{f,\mathrm{div}})'), \label{regubc}\\
&& \ub_m\in C_w([0, T]; \widetilde{\mathbf{H}}_{m,\mathrm{div}}) \cap L^2(0, T; \widetilde{\mathbf{H}}_{m,\mathrm{div}})\cap H^1(0,T; (\widetilde{\mathbf{H}}_{m,\mathrm{div}})'), \label{regubm}\\
&& \theta\in C_w([0,T]; L^2(\Omega))\cap L^2(0,T; H^1_0(\Omega))\cap H^1(0,T; (W^{1,3}_0(\Omega))'),
\eea
is called a weak solution to problem \eqref{uf1}--\eqref{IBCi2}, if the following conditions are fulfilled:

(1) For any $\mathbf{v}_f\in C^1_0((0,T); \mathbf{H}_{f,\mathrm{div}})$, $\mathbf{v}_m\in C^1_0((0,T); \widetilde{\mathbf{H}}_{m,\mathrm{div}})$ with $\mathbf{v}_f\cdot \mathbf{n}_i=\mathbf{v}_m\cdot \mathbf{n}_i$ on $\Gamma_{i}$, it holds
 \begin{eqnarray}
 &&-\int_0^T(\ub_f,\partial_t\mathbf{v}_f)_f dt
 -\varpi\int_0^T(\ub_m,\partial_t\mathbf{v}_m)_m dt\non\\
 &&
 +\int_0^T \big(\mathrm{div}(\ub_f\otimes \ub_f), \mathbf{v}_f\big)_fdt
 +2\int_0^T\big(\nu(\theta_f)\mathbb{D}(\ub_f),\mathbb{D}(\mathbf{v}_f)\big)_fdt
 \non\\
&& + \int_0^T \left(\nu(\theta_m)\mathbb{K}^{-1}\mathbf{u}_m,\mathbf{v}_m)\right)_m dt
\non\\
&&+\sum_{j=1}^{d-1} \int_0^T\!\int_{\Gamma_{i}}\frac{\alpha \nu(\theta_m)}{\sqrt{{\rm trace} (\mathbb{K})}} (\ub_f\cdot\btau_j)(\mathbf{v}_f\cdot\btau_j)dSdt\non\\
&&- \int_0^T \!\int_{\Gamma_{i}} \frac12|\ub_f|^2 (\mathbf{v}_f\cdot \mathbf{n}_{i}) dSdt  \non\\
&=&  \int_0^T (\theta_f \mathbf{k},  \mathbf{v}_f )_f dt +
\int_0^T (\theta_m \mathbf{k},  \mathbf{v}_m )_m dt.\label{weak1}
\end{eqnarray}

(2) For any $\phi\in C_0^1((0,T); W_0^{1,3}(\Omega))$, it holds
\begin{eqnarray}
&&-\int_0^T(\theta,\partial_t\phi)dt+\int_0^T(\lambda(\theta) \nabla \theta,\nabla\phi)dt=\int_0^T(\theta \ub, \nabla \phi)dt.
\label{weak3}
\end{eqnarray}

(3) $\ub_f|_{t=0}=\ub_{0f}(x)$, $\ub_m|_{t=0}=\ub_{0m}(x)$, $\theta|_{t=0}=\theta_{0}(x)$.\medskip

\textbf{Case 2}: $\varpi=0$. Consider the initial data $\ub_{0f}(x)\in \widetilde{\mathbf{H}}_{f,\mathrm{div}}$, $\theta_{0}\in L^2(\Omega)$. The triple $(\ub_f, \mathbf{u}_m, \theta)$ satisfying
\bea
&& \ub_f\in C_w([0,T]; \widetilde{\mathbf{H}}_{f,\mathrm{div}})\cap L^2(0, T; \mathbf{H}_{f,\mathrm{div}})\cap W^{1,\frac43}(0,T; (\mathbf{H}_{f,\mathrm{div}})'),\label{regubc0}\\
&& \ub_m\in L^2(0, T; \widetilde{\mathbf{H}}_{m,\mathrm{div}}),\label{regubm0}\\
&& \theta\in C_w([0,T]; L^2(\Omega))\cap L^2(0,T; H^1_0(\Omega))\cap W^{1,\frac87}(0,T; (W^{1,4}_0(\Omega))'),
\eea
is called a weak solution to problem \eqref{uf1}--\eqref{IBCi2}, if the following conditions are satisfied:

(1) For any $\mathbf{v}_f\in C^1_0((0,T); \mathbf{H}_{f,\mathrm{div}})$, $\mathbf{v}_m\in C^1_0((0,T); \widetilde{\mathbf{H}}_{m,\mathrm{div}})$ with $\mathbf{v}_f\cdot \mathbf{n}_i=\mathbf{v}_m\cdot \mathbf{n}_i$ on $\Gamma_{i}$, it holds
 \begin{eqnarray}
 &&-\int_0^T(\ub_f,\partial_t\mathbf{v}_f)_f dt
 +\int_0^T \big(\mathrm{div}(\ub_f\otimes \ub_f), \mathbf{v}_f\big)_fdt\non\\
 &&
 +2\int_0^T\big(\nu(\theta_f)\mathbb{D}(\ub_f),\mathbb{D}(\mathbf{v}_f)\big)_fdt
  + \int_0^T \left(\nu(\theta_m)\mathbb{K}^{-1}\mathbf{u}_m,\mathbf{v}_m)\right)_m dt
\non\\
&&+\sum_{j=1}^{d-1} \int_0^T\!\int_{\Gamma_{i}}\frac{\alpha \nu(\theta_m)}{\sqrt{{\rm trace} (\mathbb{K})}} (\ub_f\cdot\btau_j)(\mathbf{v}_f\cdot\btau_j)dSdt\non\\
&&- \int_0^T\! \int_{\Gamma_{i}} \frac12|\ub_f|^2 (\mathbf{v}_f\cdot \mathbf{n}_{i}) dSdt  \non\\
&=& \int_0^T (\theta_f \mathbf{k},  \mathbf{v}_f )_f dt+
\int_0^T (\theta_m \mathbf{k},  \mathbf{v}_m )_m dt.\label{weak10}
\end{eqnarray}

(2) For any $\phi\in C_0^1((0,T); W_0^{1,4}(\Omega))$, it holds
\begin{eqnarray}
&&-\int_0^T(\theta,\partial_t\phi)dt+\int_0^T(\lambda(\theta) \nabla \theta,\nabla\phi)dt=\int_0^T(\theta \ub, \nabla \phi)dt.
\label{weak30}
\end{eqnarray}

(3) $\ub_f|_{t=0}=\ub_{0f}(x)$, $\theta|_{t=0}=\theta_{0}(x)$.
  \end{definition}
 \begin{remark}
We note that the interface/boundary conditions \eqref{IBC0}--\eqref{IBCi2} are enforced as consequences of the weak formulation stated in Definition \ref{defweak}. The equivalence for smooth (or strong) solutions between the weak formulation and the classical form can be verified in a straightforward way using integration by parts. We may refer to \cite{CR08} for detailed computations about the two dimensional Navier--Stokes--Darcy system, and we mention \cite{Zha17} for the weak formulation of second order parabolic transmission problems. Here in \eqref{weak3}, \eqref{weak30} and below, we always use the convention $$\lambda(\theta)|_{\Omega_f}=\lambda_f(\theta_f),\quad \lambda(\theta)|_{\Omega_m}=\lambda_m(\theta_m)$$ for the thermal diffusivity.
\end{remark}
\smallskip
Now we are in a position to state the main results of this paper.

\bt[Existence of global weak solutions]\label{thmEx}
 Suppose that $d=2,3$, $T>0$ being arbitrary and the assumptions (A1)--(A3) are satisfied.
\begin{itemize}
 \item[(i)] If $\varpi>0$, for any $\ub_{0f}\in \widetilde{\mathbf{H}}_{f,\mathrm{div}}$, $\ub_{0m}\in\widetilde{\mathbf{H}}_{m,\mathrm{div}}$ with $\ub_{0f}\cdot\mathbf{n}_{i} = \ub_{0m}\cdot\mathbf{n}_{i}$
on $\Gamma_{i}$, and $\theta_0 \in L^2(\Omega)$, problem \eqref{uf1}--\eqref{IBCi2} admits at least one global weak solution $(\mathbf{u}_f, \ub_m, \theta)$ on $[0,T]$.

 \item[(ii)] If $\varpi=0$, for any $\ub_{0f}\in \widetilde{\mathbf{H}}_{f,\mathrm{div}}$ and $\theta_0 \in L^2(\Omega)$, problem \eqref{uf1}--\eqref{IBCi2} admits at least one global weak solution $(\mathbf{u}_f, \ub_m, \theta)$ on $[0,T]$.
 \end{itemize}
\et
\begin{remark}
The pressure terms $P_f, P_m$ will be understood in the distributional sense and can be constructed, for instance,  as in \cite[Chapter 3, Section 3.5]{Temam1977}. Besides, when $\varpi=0$, the pressure $P_m$ in the matrix part can be also regarded as a weak solution of the Neumann problem
\begin{equation*}
\begin{cases}
-\Delta P_m=\mathrm{div}(\nu(\theta_m)\mathbb{K}^{-1}\ub_m- \theta_m\mathbf{k}),\qquad\quad\ \  \text{in}\ \Omega_m,\\
\partial_{\widehat{\mathbf{n}}_m}P_m=(\nu(\theta_m)\mathbb{K}^{-1}\ub_m- \theta_m\mathbf{k})\cdot \widehat{\mathbf{n}}_m, \qquad \ \text{on}\ \partial\Omega_m.
\end{cases}
\end{equation*}
For any function $h\in H^1(\Omega)$, it holds
\begin{align*}
& \big\langle \mathrm{div}(\nu(\theta_m)\mathbb{K}^{-1}\ub_m- \theta_m\mathbf{k}),\,h\big\rangle_{(H^1(\Omega_m))',H^1(\Omega_m)}\\
&\quad \leq C(\|\ub_m\|_{\mathbf{L}^2(\Omega_m)}+\|\theta_m\|_{L^2(\Omega_m)})\|h\|_{H^1(\Omega_m)}.
\end{align*}
On the other hand, at least for the simple case when  $\nu$, $\kappa$ are positive constants, one can verify that $(\nu \mathbb{K}^{-1}\ub_m- \theta_m\mathbf{k})\cdot \widehat{\mathbf{n}}_m \in L^2(0,T;H^{-\frac12}(\partial \Omega_m))$. Then we can obtain $P_m\in L^2(0,T; H^1(\Omega_m))$.
\end{remark}

Our second result concerns the uniqueness property of solutions to problem \eqref{uf1}--\eqref{IBCi2}. More precisely, we deduce the following weak-strong uniqueness result in both two and three dimensions:

\bt[Weak-strong uniqueness]\label{thmuniq}
Suppose that $d=2,3$, $\varpi\geq 0$ and the assumptions (A1)--(A3) are satisfied.
 Let $(\ub_f,\ub_m,\theta)$, $(\bar{\ub}_f,\bar{\ub}_m,\bar{\theta})$ be two solutions to problem \eqref{uf1}--\eqref{IBCi2} on a certain time interval $[0,T]$, both emanating from the same initial data $\ub_{0f}\in \widetilde{\mathbf{H}}_{f,\mathrm{div}}$, $\varpi \ub_{0m}\in\widetilde{\mathbf{H}}_{m,\mathrm{div}}$ with $\ub_{0f}\cdot\mathbf{n}_{i} = \ub_{0m}\cdot\mathbf{n}_{i}$
on $\Gamma_{i}$ (if $\varpi>0$) and $\theta_0 \in L^2(\Omega)$. In particular, we assume that
$(\ub_f,\ub_m,\theta)$ is a global weak solution obtained in Theorem \ref{thmEx} and $(\bar{\ub}_f,\bar{\ub}_m,\bar{\theta})$ is a strong solution with the following additional regularity
\begin{equation}
\bar{\ub}_f\in L^4(0,T;\mathbf{W}^{1,6}(\Omega_f)),\   \bar{\ub}_m\in L^4(0,T;\mathbf{L}^6(\Omega_m)),\ \bar{\theta}\in L^8(0,T;W^{1,4}(\Omega)).
\label{areg}
\end{equation}
Then it holds  $$(\ub_f,\ub_m,\theta)=(\bar{\ub}_f,\bar{\ub}_m,\bar{\theta}),\quad  \text{on}\ [0,T].$$
\et

\begin{remark}
The additional regularity conditions \eqref{areg} can be weakened for some special cases. For instance, if the fluid viscosity $\nu$ is a positive constant, the condition $\bar{\ub}_f\in L^4(0,T;\mathbf{W}^{1,6}(\Omega_f))$ can be replaced by the classical condition $\bar{\ub}_f\in L^4(0,T;\mathbf{H}^{1}(\Omega_f))$ (see \cite{Be95}), and the condition $\bar{\ub}_m\in L^4(0,T;\mathbf{L}^6(\Omega_m))$ can simply be dropped. Besides, in the two dimensional case, if we assume that $\nu, \lambda_f, \lambda_m$ are all positive constants and $\theta_0\in L^\infty(\Omega)$, then one can easily check that the conclusion on uniqueness holds under just one additional regularity condition $\bar{\ub}_f\in L^\eta(0,T;\mathbf{H}^{1}(\Omega_f))$, for any $\eta>2$. Different from the case for a single homogeneous incompressible fluid in a simple domain (i.e., the classical uniqueness result for weak solutions of the Navier--Stokes equations in 2D), this additional requirement with $\eta>2$ is essentially due to the nonlinear Lions interface condition \eqref{IBCi6}.
\end{remark}


\section{Existence of Global Weak Solutions} \setcounter{equation}{0}

In this section, we prove Theorem \ref{thmEx} on the existence of global weak solutions to problem \eqref{uf1}--\eqref{IBCi2}. The proof will be given in the three dimensional case and the two dimensional case can be easily treated with minor modifications.

First, we recall an important feature of problem \eqref{uf1}--\eqref{IBCi2},  that is, it obeys a basic energy law which can lead to certain nonlinear stability of the system under suitable assumptions (see e.g., \cite{MMW19}). Denote the total energy of the coupled system by
\be
\mathcal{E}_\sigma(t)=\int_{\Omega_f}\frac{1}{2}|\ub_f|^2 dx
+ \int_{\Omega_m} \frac{\varpi}{2}|\ub_m|^2 dx
+ \int_{\Omega}\frac{\sigma}{2}|\theta|^2dx,
\label{totenergy}
\ee
for $\varpi\geq 0$ and some $\sigma>0$, where $\sigma$ is an arbitrary positive constant and it will be chosen in a suitable way below. By a similar calculation like in \cite[Section 4]{MMW19}, we have the following formal result:
\begin{lemma}[Basic energy law]\label{BEL}
Let $(\ub_m, \ub_f, \theta)$ be a smooth solution to the initial boundary value problem \eqref{uf1}--\eqref{IBCi2} on $[0,T]$. Then $(\ub_m,  \ub_f, \theta)$ satisfies the following energy equality:
\be
\frac{d}{dt} \mathcal{E}_\sigma (t)+\mathcal{D}_\sigma (t)= \mathcal{R}(t),\quad \forall\, t\in (0,T),
\label{EnergyLaw}
\ee
 where $\mathcal{E}_\sigma(t)$ is given by \eqref{totenergy} and
\bea
\mathcal{D}_\sigma (t)
&=& \int_{\Omega_f}2\nu(\theta_f)|\mathbb{D}(\ub_f)|^2dx
+\int_{\Omega_m} \nu(\theta_m)\mathbb{K}^{-1}|\ub_m|^2dx
\non\\
&&
+ \sum_{j=1}^{d-1}\int_{\Gamma_{i}} \frac{\alpha\nu(\theta_m)}{\sqrt{{\rm trace}(\mathbb{K})}}
|\ub_f\cdot\btau_j|^2 dS
+ \sigma \int_{\Omega}\lambda(\theta)|\nabla\theta|^2dx,\label{D}\\
\mathcal{R}(t)&=&\int_{\Omega_f}  (\ub_f\cdot \mathbf{k})\theta_f dx+\int_{\Omega_m}  (\ub_m\cdot \mathbf{k})\theta_m dx.
\label{R}
\eea
\end{lemma}

Inspired by \cite{ADG13,HWW}, below we apply a semi-discretization approach to prove Theorem \ref{thmEx}. First, we introduce a discrete in time, continuous in space numerical scheme for a regularized system with an approximation of Brinkman's type in the Darcy equation. The existence of weak solutions to the regularized discrete problem is then proved by using the Leray--Schauder principle. After that, we construct approximate solutions and derive uniform estimates using a discrete version of the basic energy law. Finally, by a two-step compactness argument we show the convergence of approximate solutions to a global weak solution of the original problem \eqref{uf1}--\eqref{IBCi2}.

\subsection{Time discretization of a regularized problem}
Let $\xi\in (0,1)$ be an arbitrary but fixed constant. We consider the following weak formulation of a regularized problem for the original one \eqref{uf1}--\eqref{IBCi2}:
 \begin{eqnarray}
 &&-\int_0^T(\ub_f,\partial_t\mathbf{v}_f)_f dt
 -\varpi\int_0^T(\ub_m,\partial_t\mathbf{v}_m)_m dt\non\\
 &&
 +\int_0^T \big(\mathrm{div}(\ub_f\otimes \ub_f), \mathbf{v}_f\big)_fdt
 +2\int_0^T\big(\nu(\theta_f)\mathbb{D}(\ub_f),\mathbb{D}(\mathbf{v}_f)\big)_fdt
 \non\\
&& + \int_0^T \left(\nu(\theta_m)\mathbb{K}^{-1}\mathbf{u}_m,\mathbf{v}_m)\right)_m dt +\xi \int_0^T \left(\nabla \ub_m, \nabla \mathbf{v}_m\right)_m dt
\non\\
&&+\sum_{j=1}^{d-1} \int_0^T\!\int_{\Gamma_{i}}\frac{\alpha \nu(\theta_m)}{\sqrt{{\rm trace} (\mathbb{K})}} (\ub_f\cdot\btau_j)(\mathbf{v}_f\cdot\btau_j)dSdt\non\\
&&- \int_0^T\! \int_{\Gamma_{i}} \frac12|\ub_f|^2 (\mathbf{v}_f\cdot \mathbf{n}_{i}) dSdt  \non\\
&=& \int_0^T (\theta_f \mathbf{k},  \mathbf{v}_f )_f dt +
\int_0^T (\theta_m \mathbf{k},  \mathbf{v}_m )_m dt,
\label{weak1r}
\end{eqnarray}
for any $\mathbf{v}_f\in C^1_0((0,T); \mathbf{H}_{f,\mathrm{div}})$, $\mathbf{v}_m\in C^1_0((0,T); \mathbf{H}_{m,\mathrm{div}})$ with $\mathbf{v}_f\cdot \mathbf{n}_i=\mathbf{v}_m\cdot \mathbf{n}_i$ on $\Gamma_i$, and
\begin{eqnarray}
&&-\int_0^T(\theta,\partial_t\phi)dt+\int_0^T(\lambda(\theta) \nabla \theta,\nabla\phi)dt=-\int_0^T( \ub\cdot \nabla \theta, \phi)dt,
\label{weak3r}
\end{eqnarray}
for any $\phi\in C_0^1((0,T); H_0^1(\Omega))$. Besides, the following initial conditions are satisfied:
 \begin{align}
 \ub_f|_{t=0}=\ub_{0f}(x),\quad \varpi\ub_m|_{t=0}=\varpi\ub_{0m}(x),\quad \theta|_{t=0}=\theta_{0}(x).
 \label{weakrini}
 \end{align}

We introduce a semi-implicit time discretization scheme for the regularized problem  \eqref{weak1r}--\eqref{weakrini}.
For arbitrary but fixed $T>0$ and a positive integer $N\in \mathbb{N}$, we denote by $\delta=\Delta t=\frac{T}{N}$ the size of time step. Given a triple $(\ub_f^k, \ub_m^k, \theta^k)$, $k=0,1,2,...,N-1$, our aim is to determine $(\ub_f, \ub_m, \theta)=(\ub_f^{k+1}, \ub_m^{k+1}, \theta^{k+1})$ as a solution of the following nonlinear elliptic system
\begin{eqnarray}
&&
\left(\frac{\ub_f^{k+1}-\ub_f^k}{\delta} ,\mathbf{v}_f\right)_f
+\varpi \left(\frac{\ub_m^{k+1}-\ub_m^k}{\delta} ,\mathbf{v}_m\right)_m\non\\
&&+\left(\mathrm{div}(\ub_f^{k+1}\otimes \ub_f^{k+1}), \mathbf{v}_f\right)_f
+2\left(\nu(\theta_f^k)\mathbb{D}(\ub_f^{k+1}),\mathbb{D}(\mathbf{v}_f)\right)_f\non\\
&&+\left(\nu(\theta_m^k)\mathbb{K}^{-1}\mathbf{u}_m^{k+1},\mathbf{v}_m)\right)_m +\xi \left(\nabla \ub_m^{k+1}, \mathbf{v}_m\right)_m \non\\
&& +\sum_{j=1}^{d-1} \int_{\Gamma_{i}}\frac{\alpha\nu(\theta^k_m)}{\sqrt{{\rm trace} (\mathbb{K})}} (\ub_f^{k+1}\cdot\btau_i)(\mathbf{v}_f\cdot\btau_i)dS\non\\
&& -  \int_{\Gamma_{i}} \frac12|\ub_f^{k+1}|^2 (\mathbf{v}_f\cdot \mathbf{n}_{i}) dS  \non\\ \non\\
&=& \big(\theta^{k+1}_f \mathbf{k},  \mathbf{v}_f \big)_f
 +  \big(\theta^{k+1}_m \mathbf{k},  \mathbf{v}_m \big)_m,
 \label{app1}
 \end{eqnarray}
\begin{equation}
\left(\frac{\theta^{k+1}-\theta^k}{\delta},\phi\right)
+
\big(\lambda(\theta^k)\nabla \theta^{k+1},\nabla\phi\big)
=-\big(\ub^{k+1}\cdot \nabla \theta^{k+1},  \phi\big),
   \label{app2a}
\end{equation}
for any $\mathbf{v}_f\in \mathbf{H}_{f,{\rm div}}$, $\mathbf{v}_m\in \mathbf{H}_{m,{\rm div}}$ with $\mathbf{v}_f\cdot \mathbf{n}_i=\mathbf{v}_m\cdot \mathbf{n}_i$ on $\Gamma_i$ and $\phi\in H^1_0(\Omega)$. When $\varpi=0$, we simply take $\mathbf{u}_m^0=\mathbf{0}$. In the above weak formulation for $\theta^{k+1}$, we implicitly use the modified interface condition $$\lambda_f(\theta_f^k)\frac{\partial \theta_f^{k+1}}{\partial \mathbf{n}_{i}}=\lambda_m(\theta_m^k)\frac{\partial \theta_m^{k+1}}{\partial \mathbf{n}_{i}},\quad \text{on}\ \Gamma_i.$$
In the remaining part of this subsection, we will omit the superscript $k+1$ for $\mathbf{u}_f^{k+1}$, $\mathbf{u}_m^{k+1}$, $\theta^{k+1}$ for the sake of simplicity.

The next lemma shows that the solution to problem \eqref{app1}--\eqref{app2a}, if exists, satisfies an discrete energy inequality.
\begin{lemma}[Discrete energy inequality] \label{DEE}
Suppose that $k=0,1,...,N-1$, $\ub_f^k\in \widetilde{\mathbf{H}}_{f,{\rm div}}$, $\ub_m^k\in \widetilde{\mathbf{H}}_{m,{\rm div}}$, $\theta^k\in H^1_0(\Omega)$. Let
$(\ub_f, \ub_m, \theta)\in  \mathbf{H}_{f,{\rm div}} \times \mathbf{H}_{m,{\rm div}} \times  H^1_0(\Omega)$ with  $\mathbf{u}_f\cdot \mathbf{n}_i=\mathbf{u}_m\cdot \mathbf{n}_i$ on $\Gamma_i$ be a solution to the discrete problem \eqref{app1}--\eqref{app2a}. Then
the following energy inequality holds
\begin{eqnarray}\label{DisEnLaw}
&&
  \mathcal{E}_\sigma (\ub_f,\ub_m,\theta)
  + \delta \left(\nu(\theta_f^k)\mathbb{D}(\ub_f), \mathbb{D}(\ub_f)\right)_f
  +\frac12\delta \left(\nu(\theta_m^k)\mathbb{K}^{-1}\ub_m, \ub_m\right)_m
\non\\
&& +\frac12 \delta  \sum_{j=1}^{d-1}\int_{\Gamma_{i}} \frac{\alpha\nu(\theta^k_m) }{\sqrt{{\rm trace} (\mathbb{K})}} |\ub_f\cdot\btau_j|^2 dS
+\frac12 \delta \sigma \int_\Omega \lambda(\theta^k)|\nabla\theta|^2 dx
  \non\\
  &&+\frac{1}{2}\big(\ub_f-\ub_f^k ,\ub_f-\ub_f^k\big)_f +\frac{\varpi}{2}\big(\ub_m-\ub_m^k,\ub_m-\ub_m^k\big)_m\non\\
  && +\frac{1}{2}\sigma\big(\theta-\theta^k,\theta-\theta^k \big)+ \delta \xi \left(\nabla \ub_m,\nabla \ub_m\right)_m\non\\
&\le& \mathcal{E}_\sigma \big(\ub_f^k,\ub_m^k,\theta^k\big),
\label{disBEL}
\end{eqnarray}
where $\mathcal{E}_\sigma$ is defined as in \eqref{totenergy}
with a sufficiently large constant $\sigma$ that is independent of $(\ub_f,\ub_m,\theta)$ and $(\ub_f^k,\ub_m^k,\theta^k)$.
\end{lemma}

\begin{proof}
 Taking $\mathbf{v}_f=\mathbf{u}_f$, $\mathbf{v}_m=\mathbf{u}_m$ in \eqref{app1}, using the elementary identity
\begin{align}
a\cdot(a-b)=\frac{1}{2}\left(|a|^2-|b|^2+|a-b|^2\right), \quad \forall\, a, b\in \mathbb{R} \ \ \text{or}\ \ \mathbb{R}^d,\label{eek1}
\end{align}
  we have
\begin{eqnarray}
&&\frac{1}{2\delta}\left(\ub_f,\mathbf{u}_f\right)_f+ \frac{1}{2\delta}\big(\ub_f-\ub_f^k ,\ub_f-\ub_f^k\big)_f
+
\frac{\varpi}{2\delta}\left(\ub_m,\mathbf{u}_m\right)_m\non\\
&&+ \frac{\varpi}{2\delta}\big(\ub_m-\ub_m^k ,\ub_m-\ub_m^k\big)_m+2\left(\nu(\theta_f^k)\mathbb{D}(\ub_f),\mathbb{D}(\mathbf{u}_f)\right)_f\non\\
&&+\left( \nu(\theta_m^k)\mathbb{K}^{-1}\ub_m, \ub_m\right)_m
 +\sum_{j=1}^{d-1} \int_{\Gamma_{i}}\frac{\alpha \nu(\theta^k_m)}{\sqrt{{\rm trace} (\mathbb{K})}} |\ub_f\cdot\btau_j|^2dS \non\\
 && + \xi\left(\nabla \ub_m,\nabla \ub_m\right)_m\non\\
&=&\frac{1}{2\delta}\big(\ub_f^k,\mathbf{u}_f^k\big)_f
+\frac{\varpi}{2\delta}\big(\ub_m^k,\mathbf{u}_m^k\big)_m  +(\theta_f \mathbf{k},  \mathbf{u}_f )_f
+ (\theta_m \mathbf{k},  \mathbf{u}_m )_m.
\label{eapp1}
 \end{eqnarray}
Next, taking the test function $\phi=\theta$ in \eqref{app2a}, using the boundary and interface conditions,  after integration by parts, we get
\begin{align}
& \frac{1}{2\delta}\left(\theta,\theta \right)
+\frac{1}{2\delta}\big(\theta-\theta^k,\theta-\theta^k \big)  + \int_\Omega \lambda(\theta^k)|\nabla\theta|^2 dx
=\frac{1}{2\delta}\big(\theta^k,\theta^k \big).
\label{eapp2}
\end{align}
Using Lemma \ref{equinorml}, the H\"{o}lder, Young and Poincar\'e inequalities, we estimate the last two terms on the right-hand side of \eqref{eapp1} as follows:
\begin{align}
 (\theta_f \mathbf{k},  \mathbf{u}_f )_f +  (\theta_m \mathbf{k},  \mathbf{u}_m )_m
& \leq \|\theta_f\|_{L^2(\Omega_f)}\|\mathbf{u}_f \|_{\mathbf{L}^2(\Omega_f)}+\|\theta_m\|_{L^2(\Omega_m)}\|\mathbf{u}_m \|_{\mathbf{L}^2(\Omega_m)}\non\\
& \leq C\|\theta_f\|_{L^2(\Omega_f)}\|\mathbf{u} \|_{\mathbf{Z}}+ \|\theta_m\|_{L^2(\Omega_f)}\|\mathbf{u} \|_{\mathbf{Z}}\non\\
&\leq \epsilon \|\mathbf{u} \|_{\mathbf{Z}}^2
+C\epsilon^{-1}\|\theta_f\|_{L^2(\Omega_f)}^2+\epsilon^{-1} \|\theta_m\|_{L^2(\Omega_m)}^2\non\\
&\leq \epsilon \|\mathbf{u} \|_{\mathbf{Z}}^2+ C\epsilon^{-1}\|\theta\|_{L^2(\Omega)}^2\non\\
&\leq \epsilon \|\mathbf{u} \|_{\mathbf{Z}}^2+ C\epsilon^{-1}\|\nabla \theta\|_{\mathbf{L}^2(\Omega)}^2,\non
\end{align}
for any $\epsilon>0$.
%
In view of assumptions (A1)--(A3), we can take $\epsilon$ to be sufficiently small,
for instance,
\begin{align}
\epsilon= \frac{1}{4}\min\left\{\underline{\nu},\, \underline{\nu}\overline{\kappa}^{-1},\,
\alpha\underline{\nu}\overline{\kappa}^{-\frac12}
\right\}.\label{epsi1}
\end{align}
Then multiplying \eqref{eapp2} by a sufficiently large constant $\sigma$ that depends on $\epsilon$ and adding the resultant with \eqref{eapp1}, we obtain the discrete energy inequality \eqref{DisEnLaw}.
\end{proof}

To prove the existence of solutions of the discrete problem \eqref{app1}--\eqref{app2a}, we shall adapt a fixed point  argument involving the Leray--Schauder principle (cf. \cite{ADG13} for a diffuse interface model for the two-phase flow with unmatched densities and \cite{HWW} for the Cahn--Hilliard--Stokes--Darcy system for the two-phase flow in decomposed domains). For this purpose, it will be convenient to reformulate the problem \eqref{app1}--\eqref{app2a} (again dropping the superscript $k+1$ for simplicity) as follows
\begin{eqnarray}
&&
2\left(\nu(\theta_f^k)\mathbb{D}(\ub_f),\mathbb{D}(\mathbf{v}_f)\right)_f
+\left(\nu(\theta_m^k)\mathbb{K}^{-1}\mathbf{u}_m,\mathbf{v}_m)\right)_m
+\xi \left(\nabla \ub_m,\nabla \mathbf{v}_m\right)_m\non\\
&& +\sum_{j=1}^{d-1} \int_{\Gamma_{i}}\frac{\alpha \nu(\theta^k_m)}{\sqrt{{\rm trace} (\mathbb{K})}}(\ub_f\cdot\btau_i)(\mathbf{v}_f\cdot\btau_i)dS\non\\
&=&-\left(\frac{\ub_f-\ub_f^k}{\delta} ,\mathbf{v}_f\right)_f
-\varpi \left(\frac{\ub_m-\ub_m^k}{\delta} ,\mathbf{v}_m\right)_m-\big(\mathrm{div}(\ub_f\otimes \ub_f), \mathbf{v}_f\big)_f \non\\
&&
+  \int_{\Gamma_{i}} \frac12|\ub_f|^2 (\mathbf{v}_f\cdot \mathbf{n}_{i}) dS   +(\theta_f \mathbf{k},  \mathbf{v}_f )_f
+ (\theta_m \mathbf{k},  \mathbf{v}_m )_m,
 \label{wapp1}
 \end{eqnarray}
\begin{equation}
 \big(\lambda(\theta^k)\nabla \theta,\nabla\phi\big)
=-\left(\frac{\theta-\theta^k}{\delta},\phi\right)
- ( \ub \cdot\nabla \theta, \phi).
   \label{wapp2}
\end{equation}

Define the function spaces
\begin{align}
 & \mathbf{V} =\big\{(\mathbf{u}_f,\ub_m)\ |\ \ub_f\in \mathbf{H}_{f, \mathrm{div}},\ \ub_m\in \mathbf{H}_{m,\mathrm{div}}, \ \ub_f\cdot\mathbf{n}_{i}=\ub_m\cdot\mathbf{n}_{i} \ \text{on}\ \Gamma_{i}\big\},
\label{V} \\
& \mathbf{X}=\mathbf{V}\times H_0^1(\Omega),\quad \mathbf{Y}=\mathbf{V}'\times H^{-1}(\Omega).
 \label{XY}
\end{align}
First, we introduce the operator $\mathcal{L}_k: \mathbf{V}\to \mathbf{V}'$ given by
\bea
&&\big \langle \mathcal{L}_k(\ub_f, \ub_m), (\mathbf{v}_f, \mathbf{v}_m)\big \rangle_{\mathbf{V}', \mathbf{V}}\non\\
&=&2\left(\nu(\theta_f^k)\mathbb{D}(\ub_f),\mathbb{D}(\mathbf{v}_f)\right)_f
+\left(\nu(\theta_m^k)\mathbb{K}^{-1}\mathbf{u}_m,\mathbf{v}_m)\right)_m
+\xi \left(\nabla \ub_m,\nabla \mathbf{v}_m\right)_m\non\\
&& +\sum_{j=1}^{d-1}  \int_{\Gamma_{i}}\frac{\alpha\nu(\theta^k_m)}{\sqrt{{\rm trace} (\mathbb{K})}}(\ub_f\cdot\btau_i)(\mathbf{v}_f\cdot\btau_i)dS,
\label{Lk}
\eea
for any $(\ub_f, \ub_m)$, $(\mathbf{v}_f, \mathbf{v}_m) \in \mathbf{V}$.
Using the assumptions (A1), (A3) and Lemma \ref{equinorml}, it is straightforward to verify that $\mathcal{L}_k$ is a strictly monotone, bounded
 and coercive operator on $\mathbf{V}$. Hence, it easily follows from the Lax--Milgram theorem that
\begin{lemma}\label{LLk}
Let the assumptions (A1) and (A3) be satisfied. For any given function  $\theta^k\in H^1_0(\Omega)$, the operator $\mathcal{L}_k: \mathbf{V}\to \mathbf{V}'$ is invertible and its inverse $\mathcal{L}_k^{-1}: \mathbf{V}' \to \mathbf{V}$ is continuous.
\end{lemma}

Next, we consider the operator induced by the left-hand side of \eqref{wapp2}. Define the operator
$\mathrm{div}_D: \mathbf{L}^2(\Omega)\to H^{-1}(\Omega)$ by
$ \big \langle \mathrm{div}_D \mathbf{v}, \phi\big\rangle_{H^{-1}(\Omega), H_0^1(\Omega)}=-(\mathbf{v}, \nabla \phi)$ for any $\phi \in H_0^1(\Omega)$. Then for $\lambda\in L^\infty(\Omega)$ such that $\lambda(x)\geq \underline{\lambda}>0$ almost everywhere in $\Omega$, we introduce the operator
$\mathrm{div}_D(\lambda(x) \nabla \cdot): H_0^1(\Omega)\to H^{-1}(\Omega)$ given by
$$ \big \langle \mathrm{div}_D( \lambda(x) \nabla \theta), \phi\big\rangle_{H^{-1}(\Omega), H_0^1(\Omega)}=-(\lambda(x) \nabla \theta, \nabla \phi), \quad \forall\, \phi \in H_0^1(\Omega).$$
Again, one can check that the operator $\mathrm{div}_D(\lambda(x) \nabla \cdot)$ is an isomorphism by an easy application of the Lax--Milgram theorem. Then we have
 \begin{lemma}\label{LDk}
  Let the assumption (A2) be satisfied.
  For any given $\theta^k\in H^1_0(\Omega)$, the operator
 \be
 \mathcal{M}_k:=-\mathrm{div}_D\big(\lambda(\theta^k) \nabla \cdot\big): H_0^1(\Omega)\to H^{-1}(\Omega)\label{Dk}
 \ee
  is invertible and its inverse $\mathcal{M}_k^{-1}: H^{-1}(\Omega) \to H_0^1(\Omega)$ is continuous.
\end{lemma}

Concerning the terms on the right-hand side of problem \eqref{wapp1}--\eqref{wapp2}, we consider the following operator $\mathcal{J}_k: \mathbf{X}\to \mathbf{V}'$:
 \begin{align}
 & \big\langle \mathcal{J}_k(\mathbf{w}), (\mathbf{v}_f, \mathbf{v}_m) \big\rangle_{\mathbf{V}', \mathbf{V}}\non\\
 &\quad = -\left(\frac{\ub_f-\ub_f^k}{\delta} ,\mathbf{v}_f\right)_f
-\varpi \left(\frac{\ub_m-\ub_m^k}{\delta} ,\mathbf{v}_m\right)_m\non\\
&\qquad -\left(\mathrm{div}(\ub_f\otimes \ub_f), \mathbf{v}_f\right)_f
+  \int_{\Gamma_{i}} \frac12|\ub_f|^2 (\mathbf{v}_f\cdot \mathbf{n}_{i}) dS  \non\\
&\qquad  +  (\theta_f \mathbf{k},  \mathbf{v}_f )_f
+  (\theta_m \mathbf{k},  \mathbf{v}_m )_m, \qquad \forall \, (\mathbf{v}_f, \mathbf{v}_m) \in \mathbf{V}, \label{Jk}
 \end{align}
and the operator $\mathcal{K}_k: \mathbf{X}\to H^{-1}(\Omega)$ given by
\begin{align}
& \big \langle \mathcal{K}_k(\mathbf{w}),\phi\big \rangle_{H^{-1}(\Omega),H^1_0(\Omega)} = -\left(\frac{\theta-\theta^k}{\delta}, \phi\right) - (\ub\cdot \nabla \theta , \phi),\quad \forall\, \phi\in H^1_0(\Omega),
\label{KK}
\end{align}
where we denote $\mathbf{w}=(\ub_f, \ub_m, \theta)$.

Let $\ub_f^k\in \widetilde{\mathbf{H}}_{f,\mathrm{div}}$, $\ub_m^k\in \widetilde{\mathbf{H}}_{m,\mathrm{div}}$ and $\theta^k\in H^1_0(\Omega)$  be given.
Using the above formulations \eqref{Lk}--\eqref{KK}, we now introduce the nonlinear operators $\mathcal{T}_k$, $\mathcal{G}_k: \mathbf{X}\to \mathbf{Y}$ such that
\begin{equation}
\mathcal{T}_k (\mathbf{w}) =
\left(
  \begin{array}{c}
    \mathcal{L}_k(\ub_f,\ub_m)\\
    \mathcal{M}_k (\theta)\\
  \end{array}
\right),   \label{Tk}
\end{equation}
and
\begin{equation}
\mathcal{G}_k (\mathbf{w}) =
\left(
  \begin{array}{c}
   \mathcal{J}_k(\mathbf{w}) \\
   \mathcal{K}_k(\mathbf{w})
  \end{array}
\right),  \label{Gk}               %
\end{equation}
where $\mathbf{w}=(\ub_f, \ub_m, \theta)$.
   As a consequence, denoting the solution to problem \eqref{app1}--\eqref{app2a} by $\mathbf{w}=(\ub_f, \ub_m, \theta)$, we can write \eqref{app1}--\eqref{app2a} into the following abstract form:
\begin{align}
\mathcal{T}_k(\mathbf{w})=\mathcal{G}_k(\mathbf{w}). \label{Abs}
\end{align}
\noindent
Indeed, from the above discussions, we have
\begin{proposition}\label{Abequi}
 The triple $(\ub_f, \ub_m, \theta)\in  \mathbf{X}$ is a solution of problem \eqref{app1}--\eqref{app2a} if and only if
$\mathbf{w}=(\ub_f, \ub_m, \theta) \in \mathbf{X}$ satisfies the equation
$\mathcal{T}_k(\mathbf{w})=\mathcal{G}_k(\mathbf{w})$.
\end{proposition}

We now proceed to show that the abstract equation \eqref{Abs} admits at least one solution $\mathbf{w} \in \mathbf{X}$. Recalling the definition of $\mathcal{T}_k$ and Lemmas \ref{LLk}--\ref{LDk}, we can conclude that
\begin{lemma}\label{LTG}
Let the assumptions (A1)--(A3) be satisfied. For any given function  $\theta^k\in H^1_0(\Omega)$,
    $\mathcal{T}_k: \mathbf{X}\to\mathbf{Y}$ is an invertible mapping and its inverse $\mathcal{T}^{-1}_k: \mathbf{Y}\to \mathbf{X}$ is continuous.
\end{lemma}
Next, concerning the operator $\mathcal{G}_k$, we introduce the space
$$
\widetilde{\mathbf{Y}}= \left(\big(\mathbf{H}^\frac{3}{4}(\Omega_f)\big)'
\times \mathbf{L}^2(\Omega_m)\right)\times \big(H^\frac12(\Omega)\big)',
$$
then we have
\begin{lemma}\label{LGk}
    The operator $\mathcal{G}_k: \mathbf{X}\to\widetilde{\mathbf{Y}}$ is continuous and it maps bounded sets into bounded sets.
    Moreover, the mapping $\mathcal{G}_k: \mathbf{X}\to\mathbf{Y}$ is compact.
\end{lemma}
\begin{proof}

For all $\mathbf{w}=(\ub_f, \ub_m,  \theta)\in \mathbf{X}$, using the Sobolev embedding theorems  $(d=2, 3)$, we can show that
$
\mathcal{G}_k (\mathbf{w})\in \widetilde{\mathbf{Y}}$.
Indeed, the estimates for the linear terms are obvious, thus we only need to estimate those terms that are nonlinear:
\begin{align}
\|\mathrm{div}(\ub_f\otimes \ub_f)\|
_{\big(\mathbf{H}^\frac{3}{4}(\Omega_f)\big)'}
&\leq C\|\mathrm{div}(\ub_f\otimes \ub_f)\|_{\mathbf{L}^\frac43(\Omega_f)}\non\\
&\leq C\|\mathrm{div}(\ub_f\otimes \ub_f)\|_{\mathbf{L}^\frac32(\Omega_f)}\non\\
&\leq C \|\nabla \ub_f\|_{\mathbf{L}^2(\Omega_f)}\|\ub_f\|_{\mathbf{L}^6(\Omega_f)}
\non\\
&\leq   C\|\ub_f\|_{\mathbf{H}^1(\Omega_f)}^2,\non
\end{align}
\begin{align}
\sup_{\|\mathbf{v}_f\|_{\mathbf{H}^{\frac34}(\Omega_f)}\leq 1}\left|\int_{\Gamma_{i}} |\ub_f|^2 (\mathbf{v}_f\cdot \mathbf{n}_{i}) dS\right|
&\leq \big\||\ub_f|^2\big\|_{\big(\mathbf{H}^{\frac14}(\Gamma_i)\big)'}
\|\mathbf{v}_f\|_{\mathbf{H}^{\frac14}(\Gamma_i)}\non\\
&\leq C \big\||\ub_f|^2\big\|_{\mathbf{L}^\frac{8}{5}(\Gamma_i)}
\|\mathbf{v}_f\|_{\mathbf{H}^{\frac34}(\Omega_f)}\non\\
&\leq C\|\ub_f\|^2_{\mathbf{L}^\frac{16}{5}(\Gamma_i)}
\leq C\|\ub_f\|^2_{\mathbf{L}^4(\Gamma_i)}\non\\
&\leq C\|\ub_f\|^2_{\mathbf{H}^\frac{1}{2}(\Gamma_i)}
\leq C\|\ub_f\|_{\mathbf{H}^1(\Omega_f)}^2,\non
\end{align}
\begin{align}
\|\ub\cdot\nabla \theta\|_{\big(H^\frac12(\Omega)\big)'}&\leq
\|\ub\cdot\nabla \theta\|_{L^\frac32(\Omega)}\non\\
&\leq \|\ub\|_{\mathbf{L}^6(\Omega)}\|\nabla \theta\|_{\mathbf{L}^2(\Omega)}\non\\
&\leq C\|(\ub_f,\ub_m)\|_{\mathbf{V}}\|\theta\|_{H^1_0(\Omega)}.\non
\end{align}
The second conclusion on compactness of $\mathcal{G}_k$ easily follows from the fact $\widetilde{\mathbf{Y}}\hookrightarrow\hookrightarrow \mathbf{Y}$.
\end{proof}

Since the operator $\mathcal{T}_k: \mathbf{X}\to\mathbf{Y}$ is invertible, for any $\mathbf{w}\in \mathbf{X}$ we introduce  $\mathbf{q}= \mathcal{T}_k (\mathbf{w})$ and then
the abstract equation \eqref{Abs} can be rewritten into an equivalent form such that $\mathbf{q}=(\mathcal{G}_k\circ\mathcal{T}_k^{-1})(\mathbf{q})$.
Thanks to Lemmas \ref{LTG} and \ref{LGk},
we see that the mapping $$\mathcal{N}_k\overset{\text{def}}{=}\mathcal{G}_k\circ\mathcal{T}_k^{-1}:\, \mathbf{Y}\to \mathbf{Y}$$ is indeed a compact operator,  because
$\mathcal{T}_k^{-1}$ is continuous and $\mathcal{G}_k$ is compact.

Now the original problem can be reduced to find a fixed point $\mathbf{q}$ of the operator $\mathcal{N}_k$ in $\mathbf{Y}$, that is,
\be
\mathbf{q}=\mathcal{N}_k(\mathbf{q}).\label{IaNK}
\ee
Existence of such a fixed point follows from an application of the Leray--Shauder principle. More precisely, we have

\begin{lemma}\label{degree}
Assume that assumptions (A1)--(A3) are satisfied. For any $\ub_f^k\in \widetilde{\mathbf{H}}_{f,\mathrm{div}}$, $\ub_m^k\in \widetilde{\mathbf{H}}_{m,\mathrm{div}}$ and $\theta^k\in H^1_0(\Omega)$, the abstract equation \eqref{IaNK} admits a solution  $\mathbf{q}\in \widetilde{\mathbf{Y}}\subset \mathbf{Y}$.
\end{lemma}
\begin{proof}
According to the abstract result \cite[Theorem 6.A]{Zei92}, it remains to show that there exists a constant $R>0$ such that if $\mathbf{q}\in \mathbf{Y}$ and $0\leq s\leq 1$ satisfying $\mathbf{q}=s\mathcal{N}_k(\mathbf{q})$, then $\|\mathbf{q}\|_{\mathbf{Y}}\leq R$.

We consider $\mathbf{q}\in \mathbf{Y}$ and $0\leq s\leq 1$ satisfying $\mathbf{q}=s\mathcal{N}_k(\mathbf{q})$. Denote $\mathbf{w} = \mathcal{T}_k ^{-1}(\mathbf{q})$. Then we have
$\mathcal{T}_k(\mathbf{w})=s\mathcal{G}_k(\mathbf{w})$ that  is equivalent to the weak formulation
\begin{eqnarray}
&&
2\left(\nu(\theta_f^k)\mathbb{D}(\ub_f),\mathbb{D}(\mathbf{v}_f)\right)_f
+\left(\nu(\theta_m^k)\mathbb{K}^{-1}\mathbf{u}_m,\mathbf{v}_m)\right)_m
\non\\
&& +\xi (\nabla \mathbf{u}_m,\nabla \mathbf{v}_m)_m +\sum_{j=1}^{d-1} \int_{\Gamma_{i}}\frac{\alpha \nu(\theta^k_m)}{\sqrt{{\rm trace} (\mathbb{K})}}(\ub_f\cdot\btau_i)(\mathbf{v}_f\cdot\btau_i)dS\non\\
&=&-s\left(\frac{\ub_f-\ub_f^k}{\delta} ,\mathbf{v}_f\right)_f
-s\varpi \left(\frac{\ub_m-\ub_m^k}{\delta} ,\mathbf{v}_m\right)_m\non\\
&&-s\left(\mathrm{div}(\ub_f\otimes \ub_f), \mathbf{v}_f\right)_f
+  s\int_{\Gamma_{i}} \frac12|\ub_f|^2 (\mathbf{v}_f\cdot \mathbf{n}_{i}) dS  \non\\
&& +s(\theta_f \mathbf{k},  \mathbf{v}_f )_f
+  s(\theta_m \mathbf{k},  \mathbf{v}_m )_m,
 \label{wapp1c}
 \end{eqnarray}
and
\begin{equation}
 \big(\lambda(\theta^k)\nabla \theta,\nabla\phi\big)
=-s\left(\frac{\theta-\theta^k}{\delta},\phi\right) - s(\ub\cdot \nabla \theta ,\phi).
   \label{wapp2c}
\end{equation}
Similar to the derivation of the discrete energy inequality in Lemma \ref{DEE}, we can derive
\begin{eqnarray}
&&\frac{s}{2\delta}\left(\ub_f,\mathbf{u}_f\right)_f+ \frac{s}{2\delta}\big(\ub_f-\ub_f^k ,\ub_f-\ub_f^k\big)_f
+
\frac{\varpi s}{2\delta}\left(\ub_m,\mathbf{u}_m\right)_m\non\\
&&+ \frac{\varpi s}{2\delta}\big(\ub_m-\ub_m^k ,\ub_m-\ub_m^k\big)_m+2\left(\nu(\theta_f^k)\mathbb{D}(\ub_f),\mathbb{D}(\mathbf{u}_f)\right)_f\non\\
&&+\big( \nu(\theta_m^k)\mathbb{K}^{-1}\ub_m, \ub_m\big)_m  +\xi (\nabla \mathbf{u}_m,\nabla \mathbf{u}_m)_m
\non\\
&& +\sum_{j=1}^{d-1} \int_{\Gamma_{i}}\frac{\alpha \nu(\theta^k_m)}{\sqrt{{\rm trace} (\mathbb{K})}} |\ub_f\cdot\btau_j|^2dS \non\\
&=&\frac{s}{2\delta}\big(\ub_f^k,\mathbf{u}_f^k\big)_f
+\frac{\varpi s}{2\delta}\big(\ub_m^k,\mathbf{u}_m^k\big)_m  + s(\theta_f \mathbf{k},  \mathbf{u}_f )_f + s(\theta_m \mathbf{k},  \mathbf{u}_m )_m,
\label{eapp1b}
 \end{eqnarray}
and
\begin{equation}
\frac{s}{2\delta}\left(\theta,\theta \right)
+\frac{s}{2\delta}\left(\theta-\theta^k,\theta-\theta^k \right)  + \int_\Omega \lambda(\theta^k)|\nabla\theta|^2 dx
=\frac{s}{2\delta}\left(\theta^k,\theta^k \right).
\label{eapp2b}
\end{equation}
Similar to the previous argument for \eqref{DisEnLaw}, using the fact $s\in [0,1]$, we can derive the following discrete energy inequality with respect to $s$:
\begin{eqnarray}\label{DisEnLawb}
&&
  s\mathcal{E}_\sigma (\ub_f,\ub_m,\theta)
  +\delta \underline{\nu}\left(\mathbb{D}(\ub_f), \mathbb{D}(\ub_f)\right)_f
  +\frac12\delta\underline{\nu}\bar{\kappa}^{-1} \left(\ub_m, \ub_m\right)_m
\non\\
&& +\frac12 \delta \alpha \underline{\nu}\bar{\kappa}^{-\frac12}\sum_{j=1}^{d-1}\int_{\Gamma_{i}} |\ub_f\cdot\btau_j|^2 dS
+\delta\xi (\nabla \mathbf{u}_m,\nabla \mathbf{u}_m)_m
+\frac12 \delta \sigma \underline{\lambda} \int_\Omega |\nabla\theta|^2 dx
  \non\\
&\le& s\mathcal{E}_\sigma(\ub_f^k,\ub_m^k,\theta^k)\leq \mathcal{E}_\sigma(\ub_f^k,\ub_m^k,\theta^k),
\end{eqnarray}
with a sufficiently large constant $\sigma$ that is independent of $s$.
 As a consequence, we have (recall also Lemma \ref{equinorml})
$$
\|\mathbf{w}\|_{\mathbf{X}}\leq C_k,\quad \forall\, s\in [0,1],
$$
where the constant $C_k$ may depend on $\delta$, $\xi$ and $\sigma$, but is independent of $s$.

In order to derive an estimate for $\mathbf{q}=\mathcal{T}_k(\mathbf{w})\in \mathbf{Y}$, we recall the relation $\mathbf{q} = s\mathcal{G}_k(\mathbf{w})$ and Lemma \ref{LGk} that $\mathcal{G}_k:\,\mathbf{X}\to \widetilde{\mathbf{Y}}$  maps bounded sets in $\mathbf{X}$ into bounded (compact) sets in $\mathbf{Y}$. Thus, we get $\mathbf{q}\in \widetilde{\mathbf{Y}}$ and
$$\|\mathbf{q}\|_{\mathbf{Y}} = \|s\mathcal{G}_k(\mathbf{w})\|_{\mathbf{Y}}\leq C(\|\mathbf{w}\|_{\mathbf{X}}+1)\leq R,$$
where the constant $R$ depends on $C_k$,
but is independent of $s$.

The proof of Lemma \ref{degree} is complete.
\end{proof}

As a consequence, we can conclude the existence of a weak solution to the time discrete problem \eqref{app1}--\eqref{app2a} from Lemmas \ref{DEE}, \ref{degree}  and Proposition \ref{Abequi}, that is
\begin{lemma}\label{DEEl}
Assume that assumptions (A1)--(A3) are satisfied.
For every $\ub_f^k\in \widetilde{\mathbf{H}}_{f,\mathrm{div}}$, $\ub_m^k\in \widetilde{\mathbf{H}}_{m,\mathrm{div}}$ and $\theta^k\in H^1_0(\Omega)$,  there exists a weak solution $(\ub_f, \ub_m, \theta)$ to the discrete problem \eqref{app1}--\eqref{app2a} such that
$$
 (\ub_f,\ub_m) \in \mathbf{V}, \quad \theta \in H_0^1(\Omega).
$$
Moreover, the solution $(\ub_f, \ub_m, \theta)$ satisfies the discrete energy inequality \eqref{DisEnLaw}.
\end{lemma}


\subsection{Construction of approximate solutions}

Once we have proved the existence of weak solutions to the time-discrete problem \eqref{app1}--\eqref{app2a}, we are able to construct approximate solutions to the regularized time-continuous system \eqref{weak1r}--\eqref{weak3r}. Recall that $\delta= \frac{T}{N}$, where $T>0$ and $N$ is an positive integer. We set $$t_k=k \delta, \quad k=0, 1,\cdots, N.$$
Let $(\ub_f^{k+1}, \ub_m^{k+1}, \theta^{k+1})$ ($k=0, 1,\cdots, N-1$) be chosen successively as a solution of the discrete problem \eqref{app1}--\eqref{app2a} with $(\ub_f^{k}, \ub_m^k, \theta^{k})$ being the ``initial value" (see Lemma \ref{DEEl}). In particular, we set $(\ub_f^{0}, \ub_m^{0}, \theta^{0})=(\ub_{0f}, \ub_{0m}, \theta_0)$ with the choice $\mathbf{u}_m^0=\mathbf{0}$ when $\varpi=0$.
Then for $k=0,1, \cdots, N-1$, we define the approximate solutions as follows
\begin{align*}
&\theta^\delta :=\frac{t_{k+1}-t}{\delta}\theta^{k}+\frac{t-t_k}{\delta}\theta^{k+1}, &\text{for}\ t \in [t_k, t_{k+1}],  \\
&\mathbf{u}^\delta_f:= \frac{t_{k+1}-t}{\delta}\ub_f^{k}+\frac{t-t_k}{\delta}\ub_f^{k+1}, &\text{for}\ t \in [t_k, t_{k+1}],\\
&\mathbf{u}^\delta_m:= \frac{t_{k+1}-t}{\delta}\ub_m^{k}+\frac{t-t_k}{\delta}\ub_m^{k+1}, &\text{for}\ t \in [t_k, t_{k+1}],\\
& \widehat{\ub}_f^\delta: =\ub_f^{k+1},&\text{for}\  t \in (t_k, t_{k+1}],\\
& \widehat{\ub}_m^\delta: =\ub_m^{k+1},&\text{for}\  t \in (t_k, t_{k+1}],\\
& \widehat{\mathbf{u}}^\delta|_{\Omega_f}= \widehat{\ub}_f^\delta, \quad \widehat{\mathbf{u}}^\delta|_{\Omega_m}=\widehat{\mathbf{u}}^\delta_m,
& \text{for}\  t \in (t_k, t_{k+1}],\\
&\widehat{\theta}^\delta := \theta^{k+1},  & \text{for}\  t \in (t_k, t_{k+1}],  \\
&\widetilde{\theta}^\delta := \theta^{k}, & \text{for}\  t \in [t_k, t_{k+1}).
\end{align*}
\begin{remark}
It follows from the above definitions that $\theta^\delta$, $\mathbf{u}^\delta_f$, $\mathbf{u}^\delta_m$ are continuous piecewise linear functions in time, while
 $\widehat{\ub}_f^\delta$, $\widehat{\ub}_m^\delta$, $\widehat{\theta}^\delta$ are piecewise constant (in time) functions being right continuous at the nodes
$\{t_{k+1}\}$ and  $\widetilde{\theta}^\delta$ is left continuous at the nodes $\{t_k\}$.
\end{remark}

Using the above definition of approximate solutions, we can derive from the discrete problem \eqref{app1}--\eqref{app2a} that the following identities hold:
 \begin{eqnarray}
 &&\int_0^T\big(\partial_t\ub_f^\delta,\mathbf{v}_f\big)_f dt
 +\varpi\int_0^T\big(\partial_t \ub_m^\delta,\mathbf{v}_m\big)_m dt\non\\
 &&
 +\int_0^T \big(\mathrm{div}(\widehat{\ub}_f^\delta\otimes \widehat{\ub}_f^\delta), \mathbf{v}_f\big)_fdt
 +2\int_0^T\left(\nu(\widetilde{\theta}_f^\delta)\mathbb{D}(\widehat{\ub}_f^\delta),
 \mathbb{D}(\mathbf{v}_f)\right)_fdt
 \non\\
&& + \int_0^T \left(\nu(\widetilde{\theta}_m^\delta)\mathbb{K}^{-1}\widehat{\mathbf{u}}_m^\delta,\mathbf{v}_m)\right)_m dt + \xi \int_0^T \left(\nabla \widehat{\mathbf{u}}_m^\delta,\nabla \mathbf{v}_m\right)_m dt
\non\\
&&+\sum_{j=1}^{d-1}\int_0^T\!\int_{\Gamma_{i}}
\frac{\alpha \nu(\widetilde{\theta}_m^\delta)}{\sqrt{{\rm trace} (\mathbb{K})}}
(\widehat{\ub}_f^\delta \cdot\btau_j)(\mathbf{v}_f\cdot\btau_j)dSdt\non\\
&&- \int_0^T\! \int_{\Gamma_{i}} \frac12|\widehat{\ub}_f^\delta|^2 (\mathbf{v}_f\cdot \mathbf{n}_{i}) dSdt  \non\\
&=&  \int_0^T \big(\widehat{\theta}_f^\delta \mathbf{k},  \mathbf{v}_f \big)_f dt +
\int_0^T \big(\widehat{\theta}_m^\delta \mathbf{k},  \mathbf{v}_m \big)_m dt,\label{weak1ra}
\end{eqnarray}
for any $\mathbf{v}_f\in C^1_0((0,T); \mathbf{H}_{f,\mathrm{div}})$, $\mathbf{v}_m\in C^1_0((0,T); \mathbf{H}_{m,\mathrm{div}})$ with $\mathbf{v}_f\cdot \mathbf{n}_i=\mathbf{v}_m\cdot \mathbf{n}_i$ on $\Gamma_i$, and
\begin{align}
&\int_0^T\big(\partial_t \theta^\delta,\phi\big)dt+\int_0^T\big(\lambda(\widetilde{\theta}^\delta) \nabla \widehat{\theta}^\delta,\nabla\phi\big)dt =\int_0^T\big(\widehat{\ub}^\delta \widehat{\theta}^\delta ,\nabla \phi\big)dt
\label{weak3ra}
\end{align}
for any $\phi\in C_0^1((0,T); H^1(\Omega))$.

Besides, in analogy to the estimates for  \eqref{eapp1} and \eqref{eapp2}, we can obtain the energy inequalities for $t\in [0,T]$:
\begin{eqnarray}
&&\frac{1}{2}\|\widehat{\ub}^\delta_f(t)\|_{\mathbf{L}^2(\Omega_f)}^2
+
\frac{\varpi}{2}\|\widehat{\ub}^\delta_m(t)\|_{\mathbf{L}^2(\Omega_m)}^2
+2\int_0^t\left(\nu(\widetilde{\theta}^\delta_f)\mathbb{D}(\widehat{\ub}^\delta_f),\mathbb{D}(\widehat{\mathbf{u}}^\delta_f)\right)_f d\tau\non\\
&&+\int_0^t \left( \nu(\widetilde{\theta}^\delta_m)\mathbb{K}^{-1}\widehat{\ub}^\delta_m, \widehat{\ub}^\delta_m\right)_m d\tau
 +\sum_{j=1}^{d-1}\int_0^t\! \int_{\Gamma_{i}}\frac{\alpha \nu(\widetilde{\theta}^\delta_m)}{\sqrt{{\rm trace} (\mathbb{K})}} |\widehat{\ub}^\delta_f\cdot\btau_j|^2 dS d\tau \non\\
 && + \xi\int_0^t\left(\nabla \widehat{\ub}^\delta_m,\nabla \widehat{\ub}^\delta_m\right)_m d\tau\non\\
&\leq &\frac{1}{2}\|\ub_{0f}\|_{\mathbf{L}^2(\Omega_f)}^2
+\frac{\varpi}{2}\|\ub_{0m}\|_{\mathbf{L}^2(\Omega_m)}^2 + \int_0^t \big(\widehat{\theta}^\delta_f \mathbf{k},  \widehat{\mathbf{u}}^\delta_f \big)_f + \big(\widehat{\theta}^\delta_m \mathbf{k},  \widehat{\mathbf{u}}^\delta_m \big)_m d\tau,
\label{eapp1c}
 \end{eqnarray}
and
\begin{align}
& \frac{1}{2}\|\widehat{\theta}^\delta(t)\|_{L^2(\Omega)}^2
 +\int_0^t  \int_\Omega \lambda(\widetilde{\theta}^\delta)|\nabla\widehat{\theta}^\delta|^2 dx d\tau
\leq \frac{1}{2}\|\theta_0\|_{L^2(\Omega)}^2.
\label{eapp2c}
\end{align}
Let $\mathcal{E}_\sigma^\delta(t)$ be the piecewise linear interpolation of the discrete energy $\mathcal{E}_\sigma(\ub_f^{k},\ub_m^{k}, \theta^k)$ (with the same choice for the constant  $\sigma$ as in Lemma \ref{DEE}) such that
\be \mathcal{E}_\sigma^\delta(t)=\frac{t_{k+1}-t}{\delta}\mathcal{E}_\sigma(\ub_f^{k},\ub_m^{k}, \theta^k)+\frac{t-t_k}{\delta}\mathcal{E}_\sigma(\ub_f^{k+1},\ub_m^{k+1}, \theta^{k+1}), \quad \text{for}\ t \in [t_k, t_{k+1}], \non
\ee
and $\mathcal{D}_\sigma^\delta(t)$ be the  approximate energy  dissipation
\bea
\mathcal{D}_\sigma^\delta(t)&=& 2 \left(\nu(\theta_f^k)\mathbb{D}(\ub_f^{k+1}), \mathbb{D}(\ub_f^{k+1})\right)_f+ \left( \nu(\theta_m^k)\mathbb{K}^{-1}\ub_m^{k+1}, \ub_m^{k+1}\right)_m\non\\
 && +\xi(\nabla \ub_m^{k+1}, \nabla \ub_m^{k+1})_m + \sigma \int_\Omega \lambda(\theta^k)|\nabla\theta^{k+1}|^2 dx\non\\
 && +\sum_{j=1}^{d-1}\int_{\Gamma_{i}} \frac{\alpha \nu(\theta^k_m)}{\sqrt{{\rm trace}(\mathbb{K})}}|\ub_f^{k+1}\cdot\btau_j|^2 dS, \quad \text{for}\ t \in (t_k, t_{k+1}).
  \non
\eea
We see from the discrete energy estimate \eqref{DisEnLaw} that for $k=0,1, \cdots, N-1$, it holds
\begin{align}
\frac{d}{dt}\mathcal{E}_\sigma^\delta(t)
&=\frac{1}{\delta}\left[\mathcal{E}_\sigma(\ub_f^{k+1}, \ub_m^{k+1}, \theta^{k+1})-\mathcal{E}_\sigma(\ub_f^{k}, \ub_m^{k}, \theta^k)\right]\non\\
&\leq - \frac12\mathcal{D}_\sigma^\delta(t),\quad \text{for}\ t \in (t_k, t_{k+1}).\label{bela}
\end{align}
In particular, we have for all $t\in [0,T]$,
\begin{align}
\mathcal{E}_\sigma(\widehat{\ub}_f^{\delta}(t), \widehat{\ub}_m^{\delta}(t), \widehat{\theta}^{\delta}(t))
+ \frac12\int_0^t \mathcal{D}_\sigma^\delta(\tau) d\tau\leq \mathcal{E}_\sigma(\ub_{0f},\ub_{0m},  \theta_0).\label{ee1}
\end{align}

\subsection{Proof of Theorem \ref{thmEx}}
We now proceed to finish the proof of Theorem \ref{thmEx}. First, we prove the conclusion for the case $\varpi>0$, and then we point out necessary modifications for the case $\varpi=0$.

\subsubsection{Case $\varpi>0$}

\textbf{Step 1. Passage to the limit $\delta\to 0$}.
First, we pass to the limit as $\delta\to 0$ while keeping the regularizing parameter $\xi>0$ fixed. From the energy inequality \eqref{ee1},  we can derive estimates on the approximate solutions that are uniform in $\delta$:
\begin{align}
&\|\widehat{\mathbf{u}}^\delta_f\|_{L^\infty(0,T;\mathbf{L}^2(\Omega_f))}
+\varpi^\frac12\|\widehat{\mathbf{u}}^\delta_m\|_{L^\infty(0,T;\mathbf{L}^2(\Omega_m))}
+\|\widehat{\theta}^\delta\|_{L^\infty(0,T;L^2(\Omega))} \leq C,\label{con1} \\
&\|\mathbb{D}(\widehat{\mathbf{u}}^\delta_f)\|_{L^2(0,T;\mathbf{L}^2(\Omega_f))}
+\sum_{j=1}^{d-1}\|\widehat{\ub}_f^\delta\cdot\btau_j\|_{L^2(0,T; L^2(\Gamma_{i}))}\leq C, \\
&\|\widehat{\mathbf{u}}_m^\delta\|_{L^2(0,T;\mathbf{L}^2(\Omega_m))}
+\xi^\frac12\|\nabla \widehat{\mathbf{u}}_m^\delta\|_{L^2(0,T;\mathbf{L}^2(\Omega_m))}
\leq C, \label{conm1}\\
& \|\nabla \widehat{\theta}^\delta\|_{L^2(0,T;\mathbf{L}^2(\Omega))} \leq C,\label{mues1}
\end{align}
where the constant $C$ depends on $\mathcal{E}_\sigma(\ub_{0f},\ub_{0m},\theta_0)$
and $\Omega$, but is independent of the parameters $\delta$ and $\xi$.

From the uniform estimates \eqref{con1}--\eqref{mues1} and Lemma \ref{equinorml}, we deduce that there exists a convergent subsequence $\{(\widehat{\ub}_f^\delta, \widehat{\ub}_m^\delta, \widehat{\theta}^\delta)\}$ (still denoted by the same symbols for simplicity) as $\delta \to 0$ (or equivalently $N\to +\infty$) such that
\be
\begin{cases}
\widehat{\mathbf{u}}_f^\delta \rightarrow \mathbf{u}_f^\xi & \text{ weakly star in } L^\infty(0,T; \mathbf{L}^2(\Omega_f)),\\
&\text{ weakly  in } L^2(0,T; \mathbf{H}^1(\Omega_f)),\\
 \widehat{\mathbf{u}}_m^\delta \rightarrow  \mathbf{u}_m ^\xi &\text{ weakly star in } L^\infty(0,T; \mathbf{L}^2(\Omega_m)),\\
 &\text{ weakly  in } L^2(0,T; \mathbf{H}^1(\Omega_m)),\\
\widehat{\theta}^\delta \rightarrow \theta^\xi & \text{ weakly star in } L^\infty(0,T; L^2(\Omega)), \\
& \text{ weakly  in } L^2(0,T; H^1_0(\Omega)),
\end{cases}
\label{conwmu}
\ee
for certain limit functions $(\ub_f^\xi, \ub_m^\xi, \theta^\xi)$ satisfying
\begin{align*}
& \mathbf{u}_f^\xi \in L^\infty(0, T; \widetilde{\mathbf{H}}_{f,\mathrm{div}})\cap L^2(0, T; \mathbf{H}_{f,\mathrm{div}}), \\
& \ub_m^\xi \in L^\infty(0, T; \widetilde{\mathbf{H}}_{m,\mathrm{div}})\cap L^2(0,T; \mathbf{H}_{m,\mathrm{div}}),\\
& \theta^\xi \in L^\infty(0, T; L^2(\Omega))\cap L^2(0, T; H^1_0(\Omega)),
\end{align*}
with $\ub_f^\xi\cdot\mathbf{n}_{i}=\ub_m^\xi\cdot\mathbf{n}_{i}$ on $\Gamma_i$.

In order to pass to the limit in those nonlinear terms, we need to obtain some information on the strong convergence of  $\widehat{\theta}^\delta$ (up to a subsequence). It follows from  equation \eqref{weak3ra}, the Gagliardo--Nirenberg inequality and the Sobolev embedding theorem that
\bea
&& \|\partial_t \theta^\delta\|^{\frac{4}{3}}_{L^{\frac{4}{3}}(0,T;H^{-1}(\Omega))}\non\\
& \leq& C\int_0^T \left(\|\nabla \widehat{\theta}^\delta\|_{\mathbf{L}^2(\Omega)}^\frac{4}{3}
      + \|\widehat{\theta}^\delta\|_{L^3(\Omega)}^\frac{4}{3}
      \| \widehat{\ub}^\delta\|_{\mathbf{L}^6(\Omega)}^\frac{4}{3}
      \right) dt\non\\
&\leq& C\int_0^T \|\nabla \widehat{\theta} ^\delta\|_{\mathbf{L}^2(\Omega)}^{\frac{4}{3}}dt
   + C \sup_{0\leq t\leq T} \|\widehat{\theta}^\delta\|_{L^2(\Omega)}^\frac{2}{3}
   \int_0^T  \|\widehat{\theta}^\delta\|_{H^1(\Omega)}^\frac{2}{3}\| \widehat{\ub}^\delta\|_{\mathbf{H}^1(\Omega)}^\frac{4}{3} dt\non\\
&\leq& C\int_0^T \Big(\|\nabla \widehat{\theta}^\delta\|_{\mathbf{L}^2(\Omega)}^2+1\Big)dt+ C \sup_{0\leq t\leq T} \|\widehat{\theta}^\delta\|_{L^2(\Omega)}^\frac{2}{3} \int_0^T  \Big(\|\widehat{\theta}^\delta\|_{H^1(\Omega)}^2+\|  \widehat{\ub}^\delta\|_{\mathbf{H}^1(\Omega)}^2\Big) dt\non\\
&\leq& C_T,\label{vpt1}
\eea
where the constant $C_T$ is independent of $\delta$, but may depend on $\xi$.
As a result, it follows that $\partial_t \theta^\delta \in L^{\frac43}(0,T; H^{-1}(\Omega))$ is bounded and $
\partial_t \theta^\delta \to \partial_t \theta^\xi $ in the sense of distribution. On the other hand, by the definition of $\theta^\delta$, it satisfies similar uniform estimates like those for $\widehat{\theta}^\delta$. Hence, applying Simon's compactness lemma (see e.g., \cite{Si85}), we deduce that there exists $$\theta^*\in L^2(0,T; H^{1-\beta}(\Omega))\cap C([0,T]; H^{-\beta}(\Omega)),$$
for some $\beta \in (0,\frac12)$ such that up to a subsequence,
$$ \theta^\delta \to \theta^*  \quad \text{strongly in}\ L^2(0,T; H^{1-\beta}(\Omega))\cap C([0,T]; H^{-\beta}(\Omega))\quad
\text{as}\ \delta\to 0.
$$
 Due to the uniqueness of limit (for the same convergent subsequence), we have $\theta^*=\theta^\xi$. Besides, by \eqref{conwmu} and \eqref{vpt1}, we also have $\theta^\xi\in C_w([0,T]; L^2(\Omega))$. Thus, concerning the initial datum, since by definition $\theta^\delta|_{t=0}=\theta_0$, we infer that
$$\theta^\xi|_{t=0}= \theta_0.$$
Next, since
$$\|\widehat{\theta}^\delta-\theta^\delta\|_{H^{-1}(\Omega)}
=\left\|(t_{k+1}-t)\frac{(\theta^{k+1}-\theta^k)}{\delta}\right\|_{H^{-1}(\Omega)}
\leq \delta \|\partial_t \theta^\delta\|_{H^{-1}(\Omega)}, \quad t\in (t_k, t_{k+1}],$$
for $k=0,1,...,N-1$, we infer from \eqref{vpt1} that
\be
\int_0^T \|\widehat{\theta}^\delta-\theta^\delta\|_{H^{-1}(\Omega)} ^\frac43 dt\leq \delta^\frac43 \int_0^T \|\partial_t \theta^\delta\|_{H^{-1}(\Omega)}^\frac43 dt\to 0 \quad \text{as}\ \delta \to 0,
\ee
which implies
\begin{align*}
\widehat{\theta}^\delta-\theta^\delta \to 0 \quad \text{strongly in}\ \ L^\frac43(0,T; H^{-1}(\Omega))\ \ \text{ as } \delta \rightarrow 0.
\end{align*}
Similarly, one can show that $\|\widetilde{\theta}^\delta -\theta^\delta\|_{L^\frac43(0,T; H^{-1}(\Omega))} \rightarrow 0$ as $\delta \rightarrow 0$. Therefore, the sequences $\{\theta^\delta\}$, $\{\widehat{\theta}^\delta\}$ and $\{\widetilde{\theta}^\delta\}$, if convergent, should converge to the same limit $\theta^\xi$. Besides, it follows from the above strong convergence, the uniform bounds \eqref{con1}, \eqref{mues1} and an interpolation argument that  as $\delta \rightarrow 0$,
\be
\widehat{\theta}^\delta, \widetilde{\theta}^\delta \to \theta^\xi \quad \text{strongly in}\ \ L^\frac{32}{17}(0,T;  H^\frac34(\Omega))\cap L^q(0,T; L^2(\Omega)).\label{vpt}
\ee
for any $q\in[2,+\infty)$. The strong convergence results also imply the pointwise almost everywhere convergence of $\theta^\delta$, $\widehat{\theta}^\delta$, $\widetilde{\theta}^\delta$ in $\Omega\times (0,T)$ (again up to a subsequence).

Next, using equation \eqref{weak1ra} and taking $\mathbf{v}_f\in L^{4} \big(0,T; \widehat{\mathbf{H}}_{f,\mathrm{div}})$,  $\mathbf{v}_m=\mathbf{0}$, we can deduce that
\bea
 && \|\partial_t \ub_f^\delta\|^{\frac{4}{3}}_{L^{\frac{4}{3}} \big(0,T;(\widehat{\mathbf{H}}_{f,\mathrm{div}})^\prime\big)}\non\\
&\leq&
C\int_0^T
 \left(\|\widehat{\ub}_f^\delta\|_{\mathbf{L}^3(\Omega_f)}^\frac43
 \|\nabla \widehat{\ub}_f^\delta\|_{\mathbf{L}^2(\Omega_f)}^\frac43+
 \| \widehat{\ub}_f^\delta\|_{\mathbf{H}^1(\Omega_f)}^\frac43
 +\|\widehat{\theta}_f^\delta\|_{L^2(\Omega_f)}^\frac43
 \right) dt\non\\
 &\leq&
C\int_0^T
 \left(\|\widehat{\ub}_f^\delta\|_{\mathbf{H}^1(\Omega_f)}^2
\|\widehat{\ub}_f^\delta\|_{\mathbf{L}^2(\Omega_f)}^\frac{2}{3} + \| \widehat{\ub}_f^\delta\|_{\mathbf{H}^1(\Omega_f)}^2
+ \|\widehat{\theta}^\delta_f\|_{L^2(\Omega_f)}^2 +1
 \right) dt\non\\
&\leq& C_T.\label{ut}
\eea
In a similar manner, taking test functions $\mathbf{v}_f=\mathbf{0}$ and $\mathbf{v}_m\in L^2(0,T;\widehat{\mathbf{H}}_{m,\mathrm{div}})$, we have
\bea
 && \varpi^2\|\partial_t \ub_m^\delta\|^2_{L^2(0,T;(\widehat{\mathbf{H}}_{m,\mathrm{div}})^\prime)}\non\\
&\leq&
C\int_0^T
 \left( \xi^2 \| \nabla \widehat{\ub}_m^\delta\|_{\mathbf{L}^2(\Omega_m)}^2
 +\|\widehat{\ub}_m^\delta\|_{\mathbf{L}^2(\Omega_m)}^2
 +\|\widehat{\theta}^\delta_m\|_{L^2(\Omega_m)}^2\right) dt\non\\
&\leq& C_T.\label{utm}
\eea
The constant $C_T$ in \eqref{ut} and \eqref{utm} is independent of $\delta$ and $\xi$ (recalling that $\xi\in (0,1)$). If the nonlinear term on the interface are involved, we notice that for any $\mathbf{v}_f\in L^{4} \big(0,T;\mathbf{H}_{f,\mathrm{div}})$, it holds
\begin{eqnarray}
 && \sup_{\|\mathbf{v}_f\|_{L^4(0,T; \mathbf{H}^{1}(\Omega_f))}\leq 1} \left|\int_0^T\!\int_{\Gamma_{i}} |\widehat{\ub}_f^\delta|^2 (\mathbf{v}_f\cdot \mathbf{n}_{i}) dSdt\right|\non\\
&\leq & \int_0^T \big\||\widehat{\ub}_f^\delta|^2\big\|_{\big(H_{00}^{\frac12}(\Gamma_i)\big)'}
\|\mathbf{v}_f\|_{\mathbf{H}_{00}^{\frac12}(\Gamma_i)}dt\non\\
&\leq & C \int_0^T \big\||\widehat{\ub}_f^\delta|^2\big\|_{L^\frac{4}{3}(\Gamma_i)}
\|\mathbf{v}_f\|_{\mathbf{H}^{1}(\Omega_f)} dt\non\\
&\leq & C \left(\int_0^T \|\widehat{\ub}_f^\delta\|^\frac83_{\mathbf{L}^\frac{8}{3}(\Gamma_i)} dt\right)^\frac34 \left(\int_0^T \|\mathbf{v}_f\|_{\mathbf{H}^{1}(\Omega_f)}^4 dt\right)^\frac14\non\\
&\leq &  C\left( \int_0^T \|\widehat{\ub}_f^\delta\|^\frac83_{\mathbf{H}^\frac{1}{4}(\Gamma_i)} dt\right)^\frac34\non\\
&\leq & C\left(\int_0^T \|\widehat{\ub}_f^\delta\|^\frac83_{\mathbf{H}^\frac{3}{4}(\Omega_f)} dt\right)^\frac34\non\\
&\leq & C\left(\int_0^T\|\widehat{\ub}_f^\delta\|_{\mathbf{H}^1(\Omega_f)}^2 dt\right)^\frac34
\sup_{0\leq t\leq T}\|\widehat{\ub}_f^\delta(t)\|_{\mathbf{L}^2(\Omega_f)}^\frac{1}{2}\non\\
&\leq & C.\non
\end{eqnarray}
Then in equation \eqref{weak1ra}, taking test functions $\mathbf{v}_f\in L^{4} (0,T; \mathbf{H}_{f,\mathrm{div}})$, $\mathbf{v}_m\in L^2(0,T;\mathbf{H}_{m,\mathrm{div}})$ with $\mathbf{v}_f\cdot\mathbf{n}_i=\mathbf{v}_m\cdot \mathbf{n}_i$ on $\Gamma_i$, we infer from the above estimates that
\begin{eqnarray}
\left|\int_0^T\!\int_{\Omega_f}\partial_t \ub_f^\delta \cdot \mathbf{v}_f dxdt + \int_0^T\!\int_{\Omega_m}\partial_t \ub_m^\delta \cdot \mathbf{v}_m dxdt \right| \leq C_T,\label{utd}
\end{eqnarray}
which also implies $\partial_t \ub_f^\delta\in L^\frac43(0,T; (\mathbf{H}_{f,\mathrm{div}})')$ and $\partial_t \ub_m^\delta \in L^2(0,T;(\mathbf{H}_{m,\mathrm{div}})')$.

 From the estimates \eqref{ut}, \eqref{utm} on  time derivatives, we can conclude the weak continuity property that $\ub_f^\xi\in C_w([0,T]; \widetilde{\mathbf{H}}_{f,\mathrm{div}})$, $\ub_m^\xi\in C_w([0,T]; \widetilde{\mathbf{H}}_{m,\mathrm{div}})$ and thus the initial conditions $\ub_f^\xi|_{t=0}= \ub_{0f}$,  $\ub_m^\xi|_{t=0}= \ub_{0m}$ are fulfilled.
Besides, parallel to the arguments for  $\theta^\delta$, $\widehat{\theta}^\delta$, we obtain the strong convergence as $\delta \rightarrow 0$ (up to a subsequence) such that
\begin{align}
\ub_f^\delta \to \ub_f^\xi, \quad &\text{strongly in}\ \ L^2(0,T; \mathbf{H}^{1-\beta}(\Omega_f))\cap C([0,T]; \mathbf{H}^{-\beta}(\Omega_f)), \label{vvf0}\\
 \widehat{\ub}_f^\delta-\ub_f^\delta \to \mathbf{0}, \quad &\text{strongly in}\ \ L^\frac{4}{3}(0,T; (\widehat{\mathbf{H}}_{f,\mathrm{div}})^\prime), \label{vvf1} \\
 \widehat{\ub}_f^\delta \to \ub_f^\xi, \quad &\text{strongly in}\ \ L^\frac{32}{17}(0,T; \mathbf{H}^\frac34(\Omega_f))\cap L^q(0,T; \mathbf{L}^2(\Omega_f)), \label{vvf2}\\
 \ub_m^\delta \to \ub_m^\xi, \quad &\text{strongly in}\ \ L^2(0,T; \mathbf{H}^{1-\beta}(\Omega_m))\cap C([0,T]; \mathbf{H}^{-\beta}(\Omega_m)), \label{vvm0}\\
 \widehat{\ub}_m^\delta-\ub_m^\delta \to \mathbf{0}, \quad &\text{strongly in}\ \ L^2(0,T; (\widehat{\mathbf{H}}_{m,\mathrm{div}})^\prime), \label{vvm1} \\
 \widehat{\ub}_m^\delta\to \ub_m^\xi, \quad &\text{strongly in}\ \ L^\frac{32}{17}(0,T; \mathbf{H}^\frac34(\Omega_m)) \cap L^q(0,T; \mathbf{L}^2(\Omega_m)), \label{vvm2}
\end{align}
for some $\beta\in (0,\frac12)$ and any $q\geq 2$. Hence, we can further deduce the strong convergence of nonlinear terms
\begin{eqnarray}
& &\int_0^T \|\widehat{\ub}_f^\delta\otimes \widehat{\ub}_f^\delta -
\ub_f^\xi\otimes \ub_f^\xi\|_{\mathbf{L}^{\frac43}(\Omega_f)}^\frac{32}{17} dt\non\\
&\leq &\int_0^T \|\widehat{\ub}_f^\delta-\ub_f^\xi\|_{\mathbf{L}^{4}(\Omega_f)}^\frac{32}{17}
\|\widehat{\ub}_f^\delta+\ub_f^\xi\|_{\mathbf{L}^{2}(\Omega_f)}^\frac{32}{17} dt\non\\
& \leq & \sup_{0\leq t\leq T}\|\widehat{\ub}_f^\delta(t)+\ub_f^\xi(t)\|_{\mathbf{L}^{2}(\Omega_f)}
^\frac{32}{17}\int_0^T \|\widehat{\ub}_f^\delta-\ub_f^\xi\|_{\mathbf{H}^{\frac34}(\Omega_f)}^\frac{32}{17} dt\non\\
& \to & 0\quad \text{as}\ \delta\to 0,\label{covua}
\end{eqnarray}
and
\begin{eqnarray}
& &\int_0^T
\big\||\widehat{\ub}_f^\delta|^2-|\ub_f^\xi|^2\big\|_{L^\frac{4}{3}(\Gamma_i)}^\frac{32}{29}dt\non\\
& \leq & \int_0^T \|\widehat{\ub}_f^\delta-\ub_f^\xi\|_{\mathbf{L}^\frac{8}{3}(\Gamma_i)}^\frac{32}{29}
\|\widehat{\ub}_f^\delta+\ub_f^\xi\|_{\mathbf{L}^\frac{8}{3}(\Gamma_i)}^\frac{32}{29}dt\non\\
& \leq & C\int_0^T \|\widehat{\ub}_f^\delta-\ub_f^\xi\|_{\mathbf{H}^\frac{1}{4}(\Gamma_i)}^\frac{32}{29}
\|\widehat{\ub}_f^\delta+\ub_f^\xi\|_{\mathbf{H}^\frac{1}{4}(\Gamma_i)}^\frac{32}{29}dt\non\\
& \leq & C\int_0^T \|\widehat{\ub}_f^\delta-\ub_f^\xi\|_{\mathbf{H}^\frac{3}{4}(\Omega_f)}^\frac{32}{29}
\|\widehat{\ub}_f^\delta+\ub_f^\xi\|_{\mathbf{H}^1(\Omega_f)}^\frac{24}{29}
\|\widehat{\ub}_f^\delta+\ub_f^\xi\|_{\mathbf{L}^2(\Omega_f)}^\frac{8}{29} dt\non\\
& \leq & C\left(\int_0^T \|\widehat{\ub}_f^\delta-\ub_f^\xi\|_{\mathbf{H}^\frac{3}{4}(\Omega_f)} ^\frac{32}{17}\right)^\frac{17}{29}
\left(\int_0^T \|\widehat{\ub}_f^\delta+\ub_f^\xi\|_{\mathbf{H}^1(\Omega_f)}^2dt\right)^\frac{12}{29}
\non\\
&& \quad \times \sup_{0\leq t\leq T}\|\widehat{\ub}_f^\delta(t)+\ub_f^\xi(t)\|_{\mathbf{L}^2(\Omega_f)}^\frac{8}{29}\non\\
& \to & 0 \quad \text{as}\ \delta\to 0.
\label{covub}
\end{eqnarray}

Based on the a.e. and strong convergence of $\widetilde{\theta}^\delta$, the assumptions (A1)--(A2) and the Sobolev embedding theorem, we see that
\begin{align*}
\nu(\widetilde{\theta}^\delta)\to \nu(\theta^\xi), \quad &\text{strongly in}\ \ L^2(0,T; L^4(\Omega)),\\
\nu(\widetilde{\theta}_m^\delta)\to \nu(\theta^\xi_m), \quad &\text{strongly in}\ \ L^2(0,T; L^\frac{8}{3}(\Gamma_i)),\\
\lambda(\widetilde{\theta}^\delta)\to \lambda(\theta^\xi),\quad & \text{strongly in}\ \ L^2(0,T; L^4(\Omega)).
\end{align*}
Then we deduce that
\begin{eqnarray*}
&&\left|\int_0^T\left(\nu(\widetilde{\theta}_f^\delta)\mathbb{D}(\widehat{\ub}_f^\delta),
 \mathbb{D}(\mathbf{v}_f)\right)_fdt
- \int_0^T\left(\nu(\theta_f^\xi)\mathbb{D}(\ub_f^\xi),
 \mathbb{D}(\mathbf{v}_f)\right)_fdt\right|
 \\
&=&\left|\int_0^T\left((\nu(\widetilde{\theta}_f^\delta)-\nu(\theta_f^\xi))\mathbb{D}(\widehat{\ub}_f^\delta),
 \mathbb{D}(\mathbf{v}_f)\right)_fdt\right|\\
 && 
 +\left|\int_0^T\left(\nu(\theta_f^\xi)(\mathbb{D}(\widehat{\ub}_f^\delta)-\mathbb{D}(\ub_f^\xi)),
 \mathbb{D}(\mathbf{v}_f)\right)_fdt\right|\\
 &\leq& \sup_{0\leq t\leq T}\|\mathbb{D}(\mathbf{v}_f(t))\|_{\mathbf{L}^\infty(\Omega_f)}
 \left(\int_0^T \|\nu(\widetilde{\theta}_f^\delta)-\nu(\theta_f^\xi)\|_{L^2(\Omega_f)}^2dt\right)^\frac12
 \left(\int_0^T \|\mathbb{D}(\widehat{\ub}_f^\delta)\|_{\mathbf{L}^2(\Omega_f)}^2dt\right)^\frac12\\
 && +\left|\int_0^T\left((\mathbb{D}(\widehat{\ub}_f^\delta)-\mathbb{D}(\ub_f^\xi)),
 \nu(\theta_f^\xi)\mathbb{D}(\mathbf{v}_f)\right)_fdt\right|\\
 &\to& 0\quad \text{as}\ \delta\to 0,
\end{eqnarray*}
for any $\mathbf{v}_f\in C([0,T]; \mathbf{H}_{f,\mathrm{div}}\cap \mathbf{W}^{1,\infty}(\Omega_f))$. In a similar manner, we get
\begin{align*}
\left| \int_0^T \left(\nu(\widetilde{\theta}_m^\delta)\mathbb{K}^{-1}\widehat{\mathbf{u}}_m^\delta,\mathbf{v}_m)\right)_m dt- \int_0^T \left(\nu(\theta_m^\xi)\mathbb{K}^{-1}\mathbf{u}_m^\xi,\mathbf{v}_m)\right)_m dt \right|
\to 0 \quad \text{as}\ \delta\to 0,
\end{align*}
for any $\mathbf{v}_m\in C([0,T]; \mathbf{H}_{m,\mathrm{div}}\cap \mathbf{L}^{\infty}(\Omega_m))$;
\begin{eqnarray*}
& & \left|\sum_{j=1}^{d-1}\int_0^T\!\int_{\Gamma_{i}}
\frac{\alpha\nu(\widetilde{\theta}_m^\delta)}{\sqrt{{\rm trace} (\mathbb{K})}}
(\widehat{\ub}_f^\delta \cdot\btau_j)(\mathbf{v}_f\cdot\btau_j)dSdt\right.\\
& &\quad \left. -\sum_{j=1}^{d-1}\int_0^T\!\int_{\Gamma_{i}}\frac{\alpha\nu(\theta_m^\xi)}{\sqrt{{\rm trace} (\mathbb{K})}}
(\ub_f^\xi \cdot\btau_j)(\mathbf{v}_f\cdot\btau_j)dSdt\right|\to 0 \quad \text{as}\ \delta\to 0,
\end{eqnarray*}
for any $\mathbf{v}_f\in C([0,T]; \mathbf{H}_{f,\mathrm{div}}\cap \mathbf{H}^2(\Omega_f))$; and
\begin{align*}
& \left| \int_0^T\big(\lambda(\widetilde{\theta}^\delta) \nabla \widehat{\theta}^\delta,\nabla\phi\big)dt
-
\int_0^T\big(\lambda(\theta^\xi) \nabla \theta^\xi,\nabla\phi\big)dt\right|
\to 0\quad \text{as}\ \delta\to 0,\\
& \left|\int_0^T \big(\widehat{\mathbf{u}}^\delta  \widehat{\theta}^\delta, \nabla\phi\big)dt- \int_0^T \big(\mathbf{u}^\xi  \theta^\xi, \nabla \phi\big)dt\right| \to 0\quad \text{as}\ \delta\to 0,
\end{align*}
for any $\phi \in C([0,T]; H^1_0(\Omega)\cap W^{1,\infty}(\Omega))$.

Using the above convergence results, we are able to pass to the limit as $\delta\to 0$ (up to a subsequence) in \eqref{weak1ra}--\eqref{weak3ra} to show that the triple $(\ub^\xi_f, \ub^\xi_m, \theta^\xi)$ is indeed a weak solution to the regularized system \eqref{weak1r}--\eqref{weak3r} on $[0,T]$. \medskip

\textbf{Step 2. Passage to the limit $\xi\to 0$}.
Next, we pass to the limit as $\xi\to 0$ in the weak form \eqref{weak1r}--\eqref{weak3r}. To this end, we show that $(\ub_f^\xi, \ub_m^\xi, \theta^\xi)$ fulfills some energy estimates uniform in $\xi$.
 It follows from the strong convergence results \eqref{vpt}, \eqref{vvf2} and \eqref{vvm2} that as $\delta \to 0$, for almost all $t\in (0, T)$, we have (up to a subsequence),
\bea
&& \widehat{\ub}_f^\delta(t)\to \ub_f^\xi(t), \quad \ \, \text{strongly in}\ \ \mathbf{L}^2(\Omega_f),\non\\
&& \widehat{\ub}_m^\delta(t)\to \ub_m^\xi(t), \quad \text{strongly in}\ \ \mathbf{L}^2(\Omega_m),\non\\
&& \widehat{\theta}^\delta(t) \to \theta^\xi(t),\ \quad\ \  \text{strongly in}\ \ L^2(\Omega),\non
\eea
which imply that
\be
\mathcal{E}_\sigma(\widehat{\ub}_f^{\delta}(t), \widehat{\ub}_m^{\delta}(t), \widehat{\theta}^{\delta}(t))\to \mathcal{E}_\sigma(\ub_f^\xi(t), \ub_m^\xi(t), \theta^\xi(t)),\quad \text{for a.a.} \ t\in (0, T).\non
\ee
By the estimate \eqref{ee1}, we have
\begin{align}
\mathcal{E}_\sigma(\ub_f^\xi(t), \ub_m^\xi(t), \theta^\xi(t))
&\leq \mathcal{E}_\sigma(\ub_{0f},\ub_{0m},  \theta_0),\label{ee1a}
\end{align}
for a.a. $t\in (0,T)$, and thanks to the lower semi-continuity of norms, we get
\begin{align}
\int_0^T \mathcal{D}_\sigma^\xi(\tau) d\tau\leq C \mathcal{E}_\sigma(\ub_{0f},\ub_{0m},  \theta_0),\label{ee2a}
\end{align}
where
\begin{align}
\mathcal{D}_\sigma^\xi &=2 \big(\mathbb{D}(\ub_f^{\xi}), \mathbb{D}(\ub_f^{\xi})\big)_f+ \big(\ub_m^{\xi}, \ub_m^{\xi}\big)_m +\xi\big(\nabla \ub_m^{\xi}, \nabla \ub_m^{\xi}\big)_m \non\\
 &\quad +\sum_{j=1}^{d-1}\int_{\Gamma_{i}}|\ub_f^{\xi}\cdot\btau_j|^2 dS +\sigma\int_\Omega |\nabla\theta^{\xi}|^2dx.\label{Dxi}
\end{align}
Hence, it follows from \eqref{ee1a}--\eqref{Dxi} that
\be
\begin{cases}
\mathbf{u}_f^\xi \rightarrow \mathbf{u}_f & \text{ weakly star in } L^\infty(0,T; \mathbf{L}^2(\Omega_f)),\\
&\text{ weakly  in } L^2(0,T; \mathbf{H}^1(\Omega_f)),\\
 \mathbf{u}_m^\xi \rightarrow  \mathbf{u}_m &\text{ weakly star in } L^\infty(0,T; \mathbf{L}^2(\Omega_m)),\\
 \xi \nabla \mathbf{u}_m^\xi \rightarrow \mathbf{0} &\text{ strongly  in } L^2(0,T; \mathbf{L}^2(\Omega_m)),\\
\theta^\xi \rightarrow \theta & \text{ weakly star in } L^\infty(0,T; L^2(\Omega)), \\
& \text{ weakly  in } L^2(0,T; H^1_0(\Omega)),
\end{cases}
\label{conwmua}
\ee
for certain functions $(\ub_f, \ub_m, \theta)$ satisfying
\begin{align*}
& \mathbf{u}_f \in L^\infty(0, T; \widetilde{\mathbf{H}}_{f,\mathrm{div}})\cap L^2(0, T; \mathbf{H}_{f,\mathrm{div}}), \\
& \ub_m \in L^\infty(0, T; \widetilde{\mathbf{H}}_{m,\mathrm{div}}),\\
& \theta \in L^\infty(0, T; L^2(\Omega))\cap L^2(0, T; H^1_0(\Omega)),
\end{align*}
and $\mathbf{u}_f\cdot \mathbf{n}_i=\mathbf{u}_m\cdot \mathbf{n}_i$ on $\Gamma_i$. Then by similar arguments like for \eqref{vpt1}, \eqref{ut} and \eqref{utm}, we can deduce that for any $\xi\in(0,1)$, the following estimates hold:
\bea
&& \|\partial_t \theta^\xi\|^{2}_{L^{2}(0,T;(W^{1,3}_0(\Omega))')}\non\\
& \leq& C\int_0^T \left(\|\nabla \theta^\xi\|_{\mathbf{L}^2(\Omega)}^2
      + \|\theta^\xi\|_{L^6(\Omega)}^2
      \|\ub^\xi\|_{\mathbf{L}^2(\Omega)}^2
      \right) dt\non\\
&\leq& C\int_0^T \|\nabla \theta ^\xi\|_{\mathbf{L}^2(\Omega)}^{2}dt
   + C \sup_{0\leq t\leq T} \|\ub^\xi(t)\|_{\mathbf{L}^2(\Omega)}^2
   \int_0^T  \|\theta^\xi\|_{H^1(\Omega)}^2 dt\non\\
&\leq& C_T,\label{vpt1a}
\eea
\bea
 && \|\partial_t \ub_f^\xi\|^{\frac{4}{3}}
 _{L^{\frac{4}{3}}(0,T;(\widehat{\mathbf{H}}_{f,\mathrm{div}})^\prime)}\non\\
&\leq&
C\int_0^T
 \left(\|\ub_f^\xi\|_{\mathbf{L}^3(\Omega_f)}^\frac43
 \|\nabla \ub_f^\xi\|_{\mathbf{L}^2(\Omega_f)}^\frac43+
 \| \ub_f^\xi\|_{\mathbf{H}^1(\Omega_f)}^\frac43
 \|\theta^\xi\|_{L^2(\Omega)}^\frac43
 \right) dt\non\\
&\leq& C_T,\label{uta}
\eea
\bea
 && \varpi^2\|\partial_t \ub_m^\xi\|^2_{L^2(0,T;(\widehat{\mathbf{H}}_{m,\mathrm{div}})^\prime)}
 \non\\
&\leq&
C\int_0^T
 \left[ \xi \big(\xi \| \nabla \widehat{\ub}_m^\xi\|_{\mathbf{L}^2(\Omega_m)}^2\big)
 +\|\ub_m^\xi\|_{\mathbf{L}^2(\Omega_m)}^2
 +\|\theta^\xi\|_{L^2(\Omega)}^2\right] dt\non\\
&\leq& C_T,\label{utma}
\eea
where the constant $C_T$ in the above estimates is independent of $\xi$.  Besides, similar to \eqref{utd}, we can see that  $\partial_t \ub_f^\xi\in L^\frac43(0,T; (\mathbf{H}_{f,\mathrm{div}})')$ and $\partial_t \ub_m^\xi \in L^2(0,T;(\mathbf{H}_{m,\mathrm{div}})')$ for all $\xi \in(0,1)$.

The estimates \eqref{vpt1a}--\eqref{utma} on time derivatives
imply the weak continuity in time such that $\theta\in C_w([0,T]; H^1_0(\Omega))$, $\ub_f\in C_w([0,T]; \widetilde{\mathbf{H}}_{f,\mathrm{div}})$, $\ub_m\in C_w([0,T]; \widetilde{\mathbf{H}}_{m,\mathrm{div}})$ and thus  $\theta|_{t=0}= \theta_0$,  $\ub_f|_{t=0}= \ub_{0f}$,  $\ub_m|_{t=0}= \ub_{0m}$. We also infer the strong convergence as $\xi \rightarrow 0$ (up to a subsequence) such that
\begin{align}
  \theta^\xi \to \theta,  \quad &\text{strongly in}\ L^2(0,T; H^{1-\beta}(\Omega))\cap C([0,T]; H^{-\beta}(\Omega)),\label{vvt2a}\\
 \ub_f^\xi \to \ub_f, \quad &\text{strongly in}\ \ L^2(0,T; \mathbf{H}^{1-\beta}(\Omega_f))\cap C([0,T]; \mathbf{H}^{-\beta}(\Omega_f)), \label{vvf2a}\\
  \ub_m^\xi \to \ub_m, \quad &\text{strongly in}\ \  C([0,T]; \mathbf{H}^{-\beta}(\Omega_m)), \label{vvm2a}
\end{align}
for some $\beta\in (0,\frac12)$. Hence, following exactly the same argument as for \eqref{covua}, \eqref{covub}, we can further deduce the strong convergence of nonlinear terms
\begin{align}
& \int_0^T \|\ub_f^\xi\otimes \ub_f^\xi -
\ub_f\otimes \ub_f\|_{\mathbf{L}^{\frac43}(\Omega_f)}^2 dt
\to 0\quad \text{as}\ \xi\to 0,\non
\end{align}
\begin{align}
& \int_0^T
\big\||\ub_f^\xi|^2-|\ub_f|^2\big\|_{L^\frac{4}{3}(\Gamma_i)}^\frac87dt
\to 0 \quad \text{as}\ \xi\to 0.\non
\end{align}
Besides, using assumptions (A1)--(A2) and \eqref{vvt2a}, we get
\begin{align*}
\nu(\theta^\xi)\to \nu(\theta), \quad &\text{strongly in}\ \ L^2(0,T; L^4(\Omega)),\\
\nu(\theta^\xi_m)\to \nu(\theta_m), \quad &\text{strongly in}\ \ L^2(0,T; L^\frac{8}{3}(\Gamma_i)),\\
\lambda(\theta^\xi)\to \lambda(\theta),\quad & \text{strongly in}\ \ L^2(0,T; L^4(\Omega)),
\end{align*}
which together with \eqref{vvt2a}--\eqref{vvm2a} yield
\begin{align*}
&\left|\int_0^T\left(\nu(\theta_f^\xi)\mathbb{D}(\ub_f^\xi),
 \mathbb{D}(\mathbf{v}_f)\right)_fdt
- \int_0^T\left(\nu(\theta_f)\mathbb{D}(\ub_f),
 \mathbb{D}(\mathbf{v}_f)\right)_fdt\right|
\to 0\quad \text{as}\ \xi\to 0,
\end{align*}
for any $\mathbf{v}_f\in C([0,T]; \mathbf{H}_{f,\mathrm{div}}\cap \mathbf{W}^{1,\infty}(\Omega_f))$;
\begin{align*}
&\left|\sum_{j=1}^{d-1}\int_0^T\!\int_{\Gamma_{i}}
\frac{\alpha\nu(\theta_m^\xi)}{\sqrt{{\rm trace} (\mathbb{K})}}
(\ub_f^\xi \cdot\btau_j)(\mathbf{v}_f\cdot\btau_j)dSdt\right.\\
&\quad \left.-\sum_{j=1}^{d-1}\int_0^T\!\int_{\Gamma_{i}}
\frac{\alpha\nu(\theta_m)}{\sqrt{{\rm trace} (\mathbb{K})}}
(\ub_f \cdot\btau_j)(\mathbf{v}_f\cdot\btau_j)dSdt\right|\ \to 0\quad \text{as}\ \xi\to 0,
\end{align*}
for any $\mathbf{v}_f\in C([0,T]; \mathbf{H}_{f,\mathrm{div}}\cap \mathbf{H}^2(\Omega_f))$;
\begin{align*}
&\left|\int_0^T \left(\nu(\theta_m^\xi)\mathbb{K}^{-1}\mathbf{u}_m^\xi,\mathbf{v}_m)\right)_m dt- \int_0^T \left(\nu(\theta_m)\mathbb{K}^{-1}\mathbf{u}_m,\mathbf{v}_m)\right)_m dt\right|\\
&\leq \left|\int_0^T \left((\nu(\theta_m^\xi)-\nu(\theta_m))\mathbb{K}^{-1}\mathbf{u}_m^\xi,\mathbf{v}_m)\right)_m dt\right|+\left| \int_0^T \left(\nu(\theta_m)\mathbb{K}^{-1}(\ub_m^\xi-\mathbf{u}_m),\mathbf{v}_m)\right)_m dt\right|\non\\
&\to 0\quad \text{as}\ \xi\to 0,
\end{align*}
for any $\mathbf{v}_m\in C([0,T]; \mathbf{H}_{m,\mathrm{div}}\cap \mathbf{L}^\infty(\Omega_m))$;
and
\begin{align*}
&\left|\int_0^T\left(\lambda(\theta^\xi) \nabla \theta^\xi,\nabla\phi\right)dt
-
\int_0^T\left(\lambda(\theta) \nabla \theta,\nabla\phi\right)dt\right|\to 0\quad \text{as}\ \xi\to 0,\\
& \left|\int_0^T \big(\mathbf{u}^\xi  \theta^\xi, \nabla\phi\big)dt- \int_0^T \big(\mathbf{u}  \theta, \nabla \phi\big)dt\right| \to 0\quad \text{as}\ \xi\to 0,
\end{align*}
for any $\phi \in C([0,T]; H^1_0(\Omega)\cap W^{1,\infty}(\Omega))$.

Then we are able to pass to the limit as $\xi\to 0$ (up to a subsequence) in the weak formulation  \eqref{weak1r}--\eqref{weak3r} to show that the limit triple $(\ub_f, \ub_m, \theta)$ is indeed a weak solution to the original system \eqref{weak1}--\eqref{weak3} on $[0,T]$, keeping in mind that by integration by parts, it holds
\bea
&&\int_0^T\left(\partial_t \ub_f ,\mathbf{v}_f\right)_f dt =- \int_0^T\left( \ub_f ,\partial_t \mathbf{v}_f\right)_f dt,\non\\
&&\int_0^T\left(\partial_t \ub_m ,\mathbf{v}_m\right)_m dt =- \int_0^T\left( \ub_m ,\partial_t \mathbf{v}_m\right)_m dt,\non\\
&& \int_0^T \left(\partial_t \theta,\phi\right)dt =-\int_0^T \left( \theta,\partial_t \phi\right)dt,\non
\eea
for sufficiently regular test functions that have compact support in $(0,T)$.
On the other hand, we see that for any $(\mathbf{v}_f,\mathbf{v}_m)\in L^{4} \big(0,T;\mathbf{H}_{f,\mathrm{div}}) \times L^2(0,T; \widetilde{\mathbf{H}}_{m,\mathrm{div}})$ with $\mathbf{v}_f\cdot \mathbf{n}_i=\mathbf{v}_m\cdot \mathbf{n}_i$ on $\Gamma_{i}$, it holds (cf.  \eqref{ut}--\eqref{utd})
\begin{eqnarray}
&& \left|\int_0^T(\partial_t\ub_f,\mathbf{v}_f)_f dt
 +\varpi\int_0^T(\partial_t \ub_m,\mathbf{v}_m)_m dt\right|\non\\
&\leq &
C\left[\int_0^T
 \left( \|\ub_f\|_{\mathbf{H}^1(\Omega_f)}^2
\|\ub_f\|_{\mathbf{L}^2(\Omega_f)}^\frac{2}{3} + \| \ub_f\|_{\mathbf{H}^1(\Omega_f)}^2
+ \|\theta_f\|_{L^2(\Omega_f)}^2 +1
 \right) dt\right]^\frac34\non\\
 &&\times \left(\int_0^T\|\mathbf{v}_f\|_{\mathbf{H}^1(\Omega)}^4dt\right)^\frac14 \non\\
 && + C\left[\int_0^T\left(\|\ub_m\|_{\mathbf{L}^2(\Omega_m)}^2
 +\|\theta_m\|_{L^2(\Omega_m)}^2\right)dt\right]^\frac12 \left(\int_0^T\|\mathbf{v}_m\|_{\mathbf{L}^2(\Omega)}^2dt\right)^\frac12\non\\
 &\leq& C,\non
\end{eqnarray}
which implies $\partial_t\ub_f \in L^{\frac43}(0,T; (\mathbf{H}_{f,\mathrm{div}})')$ and $ \partial_t \ub_m\in  L^2(0,T;(\widetilde{\mathbf{H}}_{m,\mathrm{div}})')$.

Finally, in view of the energy inequalities \eqref{eapp1c}, \eqref{eapp2c} and using the above convergence results we can further conclude that for almost all $t\in [0,T]$, the global weak solution satisfies
\begin{eqnarray}
&&\frac{1}{2}\|\ub_f(t)\|_{\mathbf{L}^2(\Omega_f)}^2
+
\frac{\varpi}{2}\|\ub_m(t)\|_{\mathbf{L}^2(\Omega_m)}^2
+2\int_0^t\big(\nu(\theta_f)\mathbb{D}(\ub_f),\mathbb{D}(\mathbf{u}_f)\big)_f d\tau\non\\
&&+\int_0^t \left( \nu(\theta_m)\mathbb{K}^{-1}\ub_m, \ub_m\right)_m d\tau
 +\sum_{j=1}^{d-1}\int_0^t\! \int_{\Gamma_{i}}\frac{\alpha \nu(\theta_m)}{\sqrt{{\rm trace} (\mathbb{K})}} |\ub_f\cdot\btau_j|^2 dS d\tau \non\\
&\leq &\frac{1}{2}\|\ub_{0f}\|_{\mathbf{L}^2(\Omega_f)}^2
+\frac{\varpi}{2}\|\ub_{0m}\|_{\mathbf{L}^2(\Omega_m)}^2 + \int_0^t (\theta_f \mathbf{k},  \mathbf{u}_f )_f + (\theta_m \mathbf{k},  \mathbf{u}_m )_m d\tau,
\label{eapp1d}
 \end{eqnarray}
and
\begin{align}
& \frac{1}{2}\|\theta(t)\|_{L^2(\Omega)}^2
 +\int_0^t\!  \int_\Omega \lambda(\theta)|\nabla\theta|^2 dx d\tau
\leq \frac{1}{2}\|\theta_0\|_{L^2(\Omega)}^2.
\label{eapp2d}
\end{align}

\subsubsection{Case $\varpi=0$}
For the case of $\varpi=0$, we note that only some weaker estimates are available for the velocity $\ub_m$ in the matrix part. Keeping this in mind, below we point out necessary modifications for the proof of global weak solutions.

First, for the discrete system \eqref{weak1r}--\eqref{weak3r}, comparing with \eqref{con1} and \eqref{conm1}, we now only have the estimate for
$\|\widehat{\mathbf{u}}_m^\delta\|_{L^2(0,T;\mathbf{H}^1(\Omega_m))}$, which implies
$$
\widehat{\mathbf{u}}_m^\delta \rightarrow  \mathbf{u}_m ^\xi \text{ weakly  in } L^2(0,T; \mathbf{H}^1(\Omega_m))\quad \text{as}\ \delta\to 0
$$
for some limit function $\ub_m^\xi \in L^2(0,T; \mathbf{H}_{m,\mathrm{div}})$. The sequential strong convergence of $\widehat{\mathbf{u}}_m^\delta$ is no longer available. Taking $\delta \to 0$ (while keeping $\xi>0$ fixed), we still have the uniform estimates for the limit triple $(\ub_f^\xi, \ub_m^\xi,\theta^\xi)$ that is a weak solution to the regularized system \eqref{weak1r}--\eqref{weak3r} with $\varpi=0$:
\begin{align}
\mathcal{E}_\sigma(\ub_f^\xi(t), \theta^\xi(t))
&\leq \mathcal{E}_\sigma(\ub_{0f},  \theta_0),\label{ee1b}
\end{align}
for a.a. $t\in (0,T)$ and
\begin{align}
\int_0^T \mathcal{D}_\sigma^\xi(\tau) d\tau\leq C \mathcal{E}_\sigma(\ub_{0f}, \theta_0),\label{ee2b}
\end{align}
where $ \mathcal{D}_\sigma^\xi$ is given by \eqref{Dxi}.
Then in \eqref{conwmua}, we now only have as $\xi\to 0$ (up to a subsequence)
\begin{align*}
&\mathbf{u}_m^\xi \rightarrow  \mathbf{u}_m \quad\ \text{ weakly in } L^2(0,T; \mathbf{L}^2(\Omega_m)),\\
&\xi \nabla \mathbf{u}_m^\xi \rightarrow \mathbf{0}\quad  \text{ strongly  in } L^2(0,T; \mathbf{L}^2(\Omega_m)),
\end{align*}
 for some $\mathbf{u}_m  \in L^2(0, T; \widetilde{\mathbf{H}}_{m,\mathrm{div}})$. Next, we see that under the current regularity of $\mathbf{u}_m^\xi$, it holds
\bea
&& \|\partial_t \theta^\xi\|^{\frac87}_{L^{\frac87}(0,T;(W^{1,4}_0(\Omega))')}\non\\
& \leq& C\int_0^T \left(\|\nabla \theta^\xi\|_{\mathbf{L}^2(\Omega)}^\frac87
      + \|\theta^\xi\|_{L^4(\Omega)}^\frac87
      \|\ub^\xi\|_{\mathbf{L}^2(\Omega)}^\frac87
      \right) dt\non\\
&\leq& C\int_0^T \left(\|\nabla \theta ^\xi\|_{\mathbf{L}^2(\Omega)}^{2}+1\right)dt\non\\
&&   + C \sup_{0\leq t\leq T} \|\theta^\xi(t)\|_{L^2(\Omega)}^\frac{2}{7}
   \int_0^T   \|\ub^\xi\|_{\mathbf{L}^2(\Omega)}^\frac87 \|\theta^\xi\|_{H^1(\Omega)}^\frac{6}{7} dt\non\\
   &\leq& C\int_0^T \left(\|\nabla \theta ^\xi\|_{\mathbf{L}^2(\Omega)}^{2} +1\right)dt\non\\
&&   + C \sup_{0\leq t\leq T} \|\theta^\xi(t)\|_{L^2(\Omega)}^\frac{2}{7}
   \left(\int_0^T   \|\ub^\xi\|_{\mathbf{L}^2(\Omega)}^2dt\right)^{\frac47} \left(\int_0^T\|\theta^\xi\|_{H^1(\Omega)}^2dt\right)^\frac{3}{7} \non\\
&\leq& C_T,\label{vpt1b}
\eea
where $C_T$ is independent of $\xi$.
Keeping these modifications in mind, we can pass to the limit as $\xi\to 0$ and conclude the existence of a weak solution to system \eqref{weak1}--\eqref{weak3} on $[0,T]$ by a similar argument for the case $\varpi >0$. Moreover, the energy inequalities \eqref{eapp1d}--\eqref{eapp2d} still hold (now with $\varpi=0$).

The proof of Theorem \ref{thmEx} is complete.
\hfill$\square$


\section{Weak-Strong Uniqueness}\setcounter{equation}{0}

In this section, we prove Theorem \ref{thmuniq} on the weak-strong uniqueness of solutions to problem  \eqref{uf1}--\eqref{IBCi2}.

Let $(\mathbf{u}_f, \mathbf{u}_m, \theta)$ be a weak solution to problem \eqref{uf1}--\eqref{IBCi2}. Then from the previous section, it satisfies the energy inequalities  \eqref{eapp1d}--\eqref{eapp2d} for $\varpi\geq 0$. On the other hand, the regular solution $(\bar{\ub}_f, \bar{\ub}_m, \bar{\theta})$ that assumed to exist on $[0,T]$ is allowed to be used as a test function in its weak formulation (i.e., \eqref{weak1}--\eqref{weak3}). By a direct computation,  we obtain the following energy equalities for $(\bar{\ub}_f, \bar{\ub}_m, \bar{\theta})$:
\begin{eqnarray}
&&\frac{1}{2}\|\bar{\ub}_f(t)\|_{\mathbf{L}^2(\Omega_f)}^2
+
\frac{\varpi}{2}\|\bar{\ub}_m(t)\|_{\mathbf{L}^2(\Omega_m)}^2
+2\int_0^t\left(\nu(\bar{\theta}_f)\mathbb{D}(\bar{\ub}_f),\mathbb{D}(\bar{\mathbf{u}}_f)\right)_f d\tau\non\\
&&+\int_0^t \left( \nu(\bar{\theta}_m)\mathbb{K}^{-1}\bar{\ub}_m, \bar{\ub}_m\right)_m d\tau
 +\sum_{j=1}^{d-1} \int_0^t\!\int_{\Gamma_{i}}\frac{\alpha \nu(\bar{\theta}_m)}{\sqrt{{\rm trace} (\mathbb{K})}} |\bar{\ub}_f\cdot\btau_j|^2 dS d\tau \non\\
&=&\frac{1}{2}\|\ub_{0f}\|_{\mathbf{L}^2(\Omega_f)}^2
+\frac{\varpi}{2}\|\ub_{0m}\|_{\mathbf{L}^2(\Omega_m)}^2 +\int_0^t (\bar{\theta}_f \mathbf{k},  \bar{\mathbf{u}}_f )_f + (\bar{\theta}_m \mathbf{k},  \bar{\mathbf{u}}_m )_m d\tau,
\label{eapp1e}
 \end{eqnarray}
and
\begin{align}
& \frac{1}{2}\|\bar{\theta}(t)\|_{L^2(\Omega)}^2
 +\int_0^t\!  \int_\Omega \lambda(\bar{\theta})|\nabla\bar{\theta}|^2 dx d\tau
= \frac{1}{2}\|\theta_0\|_{L^2(\Omega)}^2.
\label{eapp2e}
\end{align}
Next, in the weak formulation \eqref{weak1} for $(\mathbf{u}_f, \mathbf{u}_m)$, we take the test function  $(\mathbf{v}_f,\mathbf{v}_m)=(-\bar{\mathbf{u}}_f,-\bar{\mathbf{u}}_m)$ and perform integration by parts to get
\begin{eqnarray}
 &&
 -\big(\mathbf{u}_f(t),\bar{\mathbf{u}(t)}_f\big)_f -\varpi\big(\mathbf{u}_m(t),\bar{\mathbf{u}}_m(t)\big)_m
 -2\int_0^t\big(\nu(\theta_f)\mathbb{D}(\ub_f),\mathbb{D}(\bar{\mathbf{u}}_f)\big)_fd\tau
 \non\\
&& - \int_0^t \left(\nu(\theta_m)\mathbb{K}^{-1}\mathbf{u}_m,\bar{\mathbf{u}}_m\right)_m d\tau - \sum_{j=1}^{d-1} \int_0^t\!\int_{\Gamma_{i}}\frac{\alpha \nu(\theta_m)}{\sqrt{{\rm trace} (\mathbb{K})}} (\ub_f\cdot\btau_j)(\bar{\mathbf{u}}_f\cdot\btau_j)dSd\tau\non\\
&=& -\|\mathbf{u}_{0f}\|_{\mathbf{L}^2(\Omega_f)}^2
 -\int_0^t(\ub_f,\partial_t\bar{\mathbf{u}}_f)_f d\tau
 -\varpi\|\mathbf{u}_{0m}\|_{\mathbf{L}^2(\Omega_m)}^2 - \varpi\int_0^t(\ub_m,\partial_t\bar{\mathbf{u}}_m)_m d\tau\non\\
&& +\int_0^t (\mathrm{div}(\ub_f\otimes \ub_f), \bar{\mathbf{u}}_f)_fd\tau - \int_0^t\! \int_{\Gamma_{i}} \frac12|\ub_f|^2 (\bar{\mathbf{u}}_f\cdot \mathbf{n}_{i}) dSd\tau
\non\\
&& -
\int_0^t (\theta_f \mathbf{k},  \bar{\mathbf{u}}_f )_f d\tau -
\int_0^t (\theta_m \mathbf{k},  \bar{\mathbf{u}}_m )_m d\tau\non\\
&=& -\|\mathbf{u}_{0f}\|_{\mathbf{L}^2(\Omega_f)}^2
    -\varpi\|\mathbf{u}_{0m}\|_{\mathbf{L}^2(\Omega_m)}^2
    +\int_0^t (\mathrm{div}(\ub_f\otimes \ub_f), \bar{\mathbf{u}}_f)_fd\tau
    \non\\
&&  + \int_0^t (\mathbf{u}_f, \mathrm{div}(\bar{\ub}_f\otimes \bar{\ub}_f))_fd\tau +2\int_0^t(\nu(\bar{\theta}_f)\mathbb{D}(\ub_f),\mathbb{D}(\bar{\mathbf{u}}_f))_fd\tau
\non\\
&&- \int_0^t\! \int_{\Gamma_{i}} \frac12|\ub_f|^2 (\bar{\mathbf{u}}_f\cdot \mathbf{n}_{i}) dSd\tau - \int_0^t\! \int_{\Gamma_{i}} \frac12|\bar{\ub}_f|^2 (\mathbf{u}_f\cdot \mathbf{n}_{i}) dSd\tau
\non\\
&& + \sum_{j=1}^{d-1} \int_0^t\!\int_{\Gamma_{i}}\frac{\alpha \nu(\bar{\theta}_m)}{\sqrt{{\rm trace} (\mathbb{K})}} (\ub_f\cdot\btau_j)(\bar{\mathbf{u}}_f\cdot\btau_j)dSd\tau+ \int_0^t \left(\nu(\bar{\theta}_m)\mathbb{K}^{-1}\mathbf{u}_m,\bar{\mathbf{u}}_m)\right)_m d\tau\non\\
&& -
\int_0^t (\theta_f \mathbf{k},  \bar{\mathbf{u}}_f )_f d\tau
-
\int_0^t (\bar{\theta}_f \mathbf{k}, \mathbf{u}_f )_f d\tau
-
\int_0^t (\theta_m \mathbf{k},  \bar{\mathbf{u}}_m )_m d\tau\non\\
&& -
\int_0^t (\bar{\theta}_m \mathbf{k}, \mathbf{u}_m )_m d\tau.
\label{weak1z}
\end{eqnarray}
In the weak formulation \eqref{weak3} for $\theta$, we take the test function $\phi=-\bar{\theta}$ and obtain
\begin{eqnarray}
&&-(\theta(t), \bar{\theta}(t))
-\int_0^t(\lambda(\theta)\nabla \theta,\nabla\bar{\theta})d\tau \non\\
&=&-\|\theta_0\|_{L^2(\Omega)}^2
-\int_0^t(\theta,\partial_t\bar{\theta})d\tau
-\int_0^t(\theta \ub, \nabla \bar{\theta})d\tau \non\\
&=&-\|\theta_0\|_{L^2(\Omega)}^2 +\int_0^t(\lambda(\bar{\theta})\nabla \theta,\nabla\bar{\theta})d\tau
-\int_0^t(\theta \ub, \nabla \bar{\theta})d\tau -\int_0^t(\bar{\theta} \bar{\ub}, \nabla \theta)d\tau.
\label{weak3z}
\end{eqnarray}
Summing up the relations \eqref{eapp1d}, \eqref{eapp1e} and \eqref{weak1z}, we obtain
\begin{eqnarray}
&& \frac12\|\mathbf{u}_f(t)-\bar{\mathbf{u}}_f(t)\|_{\mathbf{L}^2(\Omega_f)}^2+
\frac{\varpi}{2}\|\mathbf{u}_m(t)-\bar{\mathbf{u}}_m(t)\|_{\mathbf{L}^2(\Omega_m)}^2
 \non\\
 && +2\int_0^t\big(\nu(\theta_f)\mathbb{D}(\ub_f-\bar{\ub}_f),\mathbb{D}(\ub_f-\bar{\mathbf{u}}_f)\big)_f d\tau\non\\
  && +\int_0^t \big( \nu(\theta_m)\mathbb{K}^{-1}(\ub_m-\bar{\ub}_m), (\ub_m-\bar{\ub}_m)\big)_m d\tau\non\\
 && +\sum_{j=1}^{d-1} \int_0^t\!\int_{\Gamma_{i}}\frac{\alpha \nu(\theta_m)}{\sqrt{{\rm trace} (\mathbb{K})}} |(\ub_f-\bar{\ub}_f)\cdot\btau_j|^2 dS d\tau \non\\
 &\leq& I_1+I_2+I_3+I_4+I_5+I_6, \label{adiff1}
\end{eqnarray}
where
\begin{eqnarray}
I_1 &=& \int_0^t \big(\mathrm{div}(\ub_f\otimes \ub_f), \bar{\mathbf{u}}_f\big)_fd\tau + \int_0^t \big(\mathbf{u}_f, \mathrm{div}(\bar{\ub}_f\otimes \bar{\ub}_f)\big)_fd\tau,\non\\
I_2&=& - \int_0^t\! \int_{\Gamma_{i}} \frac12|\ub_f|^2 (\bar{\mathbf{u}}_f\cdot \mathbf{n}_{i}) dSd\tau - \int_0^t\! \int_{\Gamma_{i}} \frac12|\bar{\ub}_f|^2 (\mathbf{u}_f\cdot \mathbf{n}_{i}) dSd\tau,\non\\
I_3 &=& \int_0^t \big((\theta_f-\bar{\theta}_f) \mathbf{k},  \ub_f-\bar{\mathbf{u}}_f \big)_fd\tau +\int_0^t \big(\theta_m-\bar{\theta}_m) \mathbf{k}, \ub_m- \bar{\mathbf{u}}_m \big)_m d\tau,\non\\
I_4 &=& -2\int_0^t\Big((\nu(\theta_f)-\nu(\bar{\theta}_f)) (\mathbb{D}(\ub_f)-\mathbb{D}(\bar{\ub}_f)), \mathbb{D}(\bar{\mathbf{u}}_f)\Big)_fd\tau,\non \\
I_5 &=& - \int_0^t \Big((\nu(\theta_m)-\nu(\bar{\theta}_m)) \mathbb{K}^{-1}(\mathbf{u}_m-\bar{\mathbf{u}}_m),\bar{\mathbf{u}}_m\Big)_m d\tau,\non\\
I_6 &=& - \sum_{j=1}^{d-1} \int_0^t\!\int_{\Gamma_{i}}\frac{\alpha (\nu(\theta_m)- \nu(\bar{\theta}_m)) }{\sqrt{{\rm trace} (\mathbb{K})}} \big((\ub_f-\bar{\mathbf{u}}_f)\cdot\btau_j\big) (\bar{\mathbf{u}}_f\cdot\btau_j)dSd\tau.\non
\end{eqnarray}
Using the facts $\mathrm{div}\ub_f=\mathrm{div}\bar{\ub}_f=0$ and integration by parts, we have
\begin{eqnarray}
&&  \big(\mathrm{div}(\ub_f\otimes \ub_f), \bar{\mathbf{u}}_f\big)_f + \big(\mathbf{u}_f, \mathrm{div}(\bar{\ub}_f\otimes \bar{\ub}_f)\big)_f\non\\
&=&  \int_{\Omega_f} (\ub_f \cdot \nabla \ub_f)\cdot   \bar{\ub}_f dx +\left[- \int_{\Omega_f} (\bar{\ub}_f\otimes \bar{\ub}_f):\nabla \ub_f dx +\int_{\Gamma_i}(\bar{\ub}_f\otimes \bar{\ub}_f)\mathbf{n}_i\cdot \ub_f dS \right]\non\\
&& +\left[ - \int_{\Omega_f} \big[ (\ub_f-\bar{\ub}_f)\cdot \nabla \bar{\ub}_f\big]\cdot \bar{\ub}_f dx+\frac12\int_{\Gamma_i}\big[(\ub_f-\bar{\ub}_f)\cdot \mathbf{n}_i\big]|\bar{\ub}_f|^2dS\right]\non\\
&=& \int_{\Omega_f} \big[(\ub_f-\bar{\ub}_f)\cdot \nabla (\ub_f-\bar{\ub}_f) \big]\cdot \bar{\ub}_f dx +\int_{\Gamma_i}(\bar{\ub}_f\cdot  \ub_f) (\bar{\ub}_f\cdot \mathbf{n}_i) dS\non\\
&& +\frac12\int_{\Gamma_i}\big[(\ub_f-\bar{\ub}_f)\cdot \mathbf{n}_i\big]|\bar{\ub}_f|^2dS.\non
\end{eqnarray}
Then we can deduce that
\begin{eqnarray}
I_1+I_2 &=& \int_0^t\!\int_{\Omega_f} \big[(\ub_f-\bar{\ub}_f)\cdot \nabla (\ub_f-\bar{\ub}_f)\big]\cdot  \bar{\ub}_f dxd\tau \non\\
&& - \frac12 \int_0^t\!\int_{\Gamma_i} |\ub_f-\bar{\ub}_f|^2(\bar{\ub}_f\cdot\mathbf{n}_i) dSd\tau\non\\
&\leq & \int_0^t \|\ub_f-\bar{\ub}_f\|_{\mathbf{L}^3(\Omega_f)} \|\bar{\ub}_f\|_{\mathbf{L}^6(\Omega_f)} \|\nabla (\ub_f-\bar{\ub}_f)\|_{\mathbf{L}^2(\Omega_f)}  d\tau\non\\
&& +\frac12\int_0^t\| \ub_f-\bar{\ub}_f\|_{\mathbf{L}^\frac{8}{3}(\Gamma_i)}^2\|\bar{\ub}_f\|_{\mathbf{L}^4(\Gamma_i)}d\tau \non\\
&\leq & \frac{\epsilon}{2}\int_0^t\|\ub_f-\bar{\ub}_f\|_{\mathbf{H}^1(\Omega_f)}^2d\tau +\frac{C}{\epsilon}\int_0^t\|\bar{\ub}_f\|_{\mathbf{H}^1(\Omega_f)}^4 \|\ub_f-\bar{\ub}_f\|_{\mathbf{L}^2(\Omega_f)}^2d\tau \non\\
&& +C \int_0^t\| \ub_f-\bar{\ub}_f\|_{\mathbf{H}^\frac{1}{4}(\Gamma_i)}^2 \|\bar{\ub}_f\|_{\mathbf{H}^\frac12 (\Gamma_i)}d\tau \non\\
&\leq & \frac{\epsilon}{2}\int_0^t\|\ub_f-\bar{\ub}_f\|_{\mathbf{H}^1(\Omega_f)}^2d\tau +\frac{C}{\epsilon}\int_0^t\|\bar{\ub}_f\|_{\mathbf{H}^1(\Omega_f)}^4 \|\ub_f-\bar{\ub}_f\|_{\mathbf{L}^2(\Omega_f)}^2d\tau \non\\
&& + C \int_0^t\| \ub_f-\bar{\ub}_f\|_{\mathbf{H}^\frac{3}{4}(\Omega_f)}^2\|\bar{\ub}_f\|_{\mathbf{H}^1(\Omega_f)}d\tau\non\\
&\leq& \epsilon \int_0^t\|\ub_f-\bar{\ub}_f\|_{\mathbf{H}^1(\Omega_f)}^2 d\tau + \frac{C}{\epsilon} \int_0^t \|\bar{\ub}_f\|_{\mathbf{H}^1(\Omega_f)}^4 \|\ub_f-\bar{\ub}_f\|_{\mathbf{L}^2(\Omega_f)}^2d\tau,
\label{I1I2}
\end{eqnarray}
where $\epsilon>0$ is a small constant to be chosen later.
Concerning $I_3$, by Lemma \ref{equinorml}, we have
\begin{eqnarray}
I_3&\leq& \int_0^t \|\theta_f-\bar{\theta}_f\|_{L^2(\Omega_f)}\|  \ub_f-\bar{\mathbf{u}}_f\|_{\mathbf{L}^2(\Omega_f)}+ \|\theta_m-\bar{\theta}_m\|_{L^2(\Omega_m)}\| \ub_m- \bar{\mathbf{u}}_m \|_{\mathbf{L}^2(\Omega_m)} d\tau\non\\
&\leq & \frac{\epsilon}{2}\int_0^t \left(
\|  \ub_f-\bar{\mathbf{u}}_f\|_{\mathbf{L}^2(\Omega_f)} ^2  +
 \| \ub_m- \bar{\mathbf{u}}_m \|_{\mathbf{L}^2(\Omega_m)}^2 \right) d\tau  \non\\
&& +\frac{1}{2\epsilon} \int_0^t \left( \|\theta_f-\bar{\theta}_f\|_{L^2(\Omega_f)}^2+ \|\theta_m-\bar{\theta}_m\|_{L^2(\Omega_m)}^2\right) d\tau\non\\
&\leq& C\epsilon \int_0^t
\|  \ub_f-\bar{\mathbf{u}}_f\|_{\mathbf{Z}} ^2d\tau  +\frac{1}{2\epsilon} \int_0^t \|\theta-\bar{\theta}\|_{L^2(\Omega)}^2 d\tau.
\label{I3}
\end{eqnarray}
The terms $I_4,I_5,I_6$ are due to the temperature dependent viscosity and can be estimated as follows:
\begin{eqnarray}
I_4&\leq & C\int_0^t \|\nu(\theta_f)-\nu(\bar{\theta}_f)\|_{L^3(\Omega_f)} \|\mathbb{D}(\ub_f)-\mathbb{D}(\bar{\ub}_f)\|_{\mathbf{L}^2(\Omega_f)} \|\mathbb{D}(\bar{\mathbf{u}}_f)\|_{\mathbf{L}^6(\Omega_f)}d\tau \non \\
&\leq & C\int_0^t \|\theta_f-\bar{\theta}_f\|_{L^2(\Omega_f)}^\frac12 \|\nabla (\theta_f-\bar{\theta}_f)\|_{\mathbf{L}^2(\Omega_f)} ^\frac12 \|\mathbb{D}(\ub_f)-\mathbb{D}(\bar{\ub}_f)\|_{\mathbf{L}^2(\Omega_f)} \|\mathbb{D}(\bar{\mathbf{u}}_f)\|_{\mathbf{L}^6(\Omega_f)}d\tau\non\\
&\leq & \epsilon \int_0^t \|\mathbb{D}(\ub_f)-\mathbb{D}(\bar{\ub}_f)\|_{\mathbf{L}^2(\Omega_f)}^2 d\tau + \epsilon \int_0^t \|\nabla (\theta_f-\bar{\theta}_f)\|_{\mathbf{L}^2(\Omega_f)}^2 d\tau\non\\
&&  + \frac{C}{\epsilon^3} \int_0^t \|\mathbb{D}(\bar{\mathbf{u}}_f)\|_{\mathbf{L}^6(\Omega_f)}^4 \|\theta_f-\bar{\theta}_f\|_{L^2(\Omega_f)}^2 d\tau.
\label{I4}
\end{eqnarray}
In a similar manner, we get
\begin{eqnarray}
I_5 &\leq & C\int_0^t \|\nu(\theta_m)-\nu(\bar{\theta}_m)\|_{L^3(\Omega_m)} \|\mathbf{u}_m-\bar{\mathbf{u}}_m\|_{\mathbf{L}^2(\Omega_m)} \|\bar{\mathbf{u}}_m\|_{\mathbf{L}^6(\Omega)}d\tau \non\\
&\leq & \epsilon\int_0^t \|\mathbf{u}_m-\bar{\mathbf{u}}_m\|_{\mathbf{L}^2(\Omega_m)}^2d\tau + \epsilon \int_0^t \|\nabla (\theta_m-\bar{\theta}_m)\|_{\mathbf{L}^2(\Omega_m)}^2 d\tau\non\\
&&  + \frac{C}{\epsilon^3} \int_0^t \|\bar{\mathbf{u}}_m\|_{\mathbf{L}^6(\Omega_m)}^4 \|\theta_m-\bar{\theta}_m\|_{L^2(\Omega_m)}^2 d\tau.\label{I5}
\end{eqnarray}
By the trace theorem, the Gagliardo--Nirenberg inequality and Young's inequality, we deduce that
\begin{eqnarray}
I_6 &\leq & C\int_0^t \|\nu(\theta_m)- \nu(\bar{\theta}_m)\|_{L^\frac{8}{3}(\Gamma_i)} \|\ub_f-\bar{\mathbf{u}}_f\|_{\mathbf{L}^4(\Gamma_i)} \|\bar{\mathbf{u}}_f\|_{\mathbf{L}^\frac{8}{3}(\Gamma_i)} d\tau\non\\
&\leq& C \int_0^t \|\theta_m-\bar{\theta}_m\|_{H^\frac{3}{4}(\Omega_m)} \|\ub_f-\bar{\mathbf{u}}_f\|_{\mathbf{H}^1(\Omega_f)} \|\bar{\mathbf{u}}_f\|_{\mathbf{H}^\frac{3}{4}(\Omega_f)}d\tau \non\\
&\leq & \epsilon \int_0^t \|\ub_f-\bar{\mathbf{u}}_f\|_{\mathbf{H}^1(\Omega_f)}^2 d\tau + \epsilon\int_0^t \|\theta_m-\bar{\theta}_m\|_{H^1(\Omega_m)}^2d\tau \non\\
&& +\frac{C}{\epsilon^7}\sup_{0\leq t\leq T}\|\bar{\ub}_f\|_{\mathbf{L}^2(\Omega_f)}^5 \int_0^t \|\bar{\mathbf{u}}_f\|_{\mathbf{W}^{1,6}(\Omega_f)}^3 \|\theta_m-\bar{\theta}_m\|_{L^2(\Omega_m)}^2 d\tau.\label{I6}
\end{eqnarray}
Next, we sum the relations \eqref{eapp2d}, \eqref{eapp2e} and \eqref{weak3z} to get
\begin{eqnarray}
&& \frac{1}{2}\|\theta(t)-\bar{\theta}(t)\|_{L^2(\Omega)}^2
 +\int_0^t\!  \int_\Omega \lambda(\theta)|\nabla(\theta-\bar{\theta})|^2 dx d\tau
\leq  I_7+I_8,\label{adiff2}
\end{eqnarray}
where
\begin{eqnarray}
I_7&=& -\int_0^t \int_\Omega (\lambda(\theta)-\lambda(\bar{\theta}))\nabla (\theta-\bar{\theta})\cdot \nabla \bar{\theta}dx d\tau,\non\\
I_8&=& -\int_0^t(\theta \ub, \nabla \bar{\theta})d\tau -\int_0^t(\bar{\theta} \bar{\ub}, \nabla \theta)d\tau.\non
\end{eqnarray}
It follows from assumption (A2) that
\begin{eqnarray}
I_7&\leq& C\int_0^t\|\theta -\bar{\theta}\|_{L^4(\Omega)}\|\nabla (\theta-\bar{\theta})\|_{\mathbf{L}^2(\Omega)}\| \nabla \bar{\theta}\|_{\mathbf{L}^4(\Omega)}d\tau \non\\
&\leq & C\int_0^t\|\theta -\bar{\theta}\|_{L^2(\Omega)}^\frac14\|\nabla (\theta-\bar{\theta})\|_{\mathbf{L}^2(\Omega)}^\frac{7}{4} \| \nabla \bar{\theta}\|_{\mathbf{L}^4(\Omega)}d\tau \non\\
&\leq & \epsilon \int_0^t \|\nabla (\theta-\bar{\theta})\|_{\mathbf{L}^2(\Omega)}^2 d\tau + \frac{C}{\epsilon^7}\int_0^t \| \nabla \bar{\theta}\|_{\mathbf{L}^4(\Omega)}^8\|\theta -\bar{\theta}\|_{L^2(\Omega)}^2 d\tau. \label{I7}
\end{eqnarray}
Then using the incompressibility condition, integration by parts, the Gagliardo--Nirenberg inequality and Lemma \ref{equinorml}, we obtain
\begin{eqnarray}
I_8&=& -\int_0^t (\theta(\ub-\bar{\ub}),\nabla \bar{\theta})d\tau  - \int_0^t\big[(\bar{\theta} \bar{\ub},\nabla \theta)+(\theta\bar{\ub},\nabla \bar{\theta})\big]d\tau\non\\
&=& -\int_0^t \big((\theta-\bar{\theta})(\ub-\bar{\ub}),\nabla \bar{\theta}\big)d\tau
-\int_0^t \big( (\theta-\bar{\theta})\bar{\ub},\nabla(\theta-\bar{\theta})\big)d\tau \non\\
&=& -\int_0^t \big((\theta-\bar{\theta})(\ub-\bar{\ub}),\nabla \bar{\theta}\big)d\tau\non\\
&\leq& \int_0^t \| \ub- \bar{\mathbf{u}} \|_{\mathbf{L}^2(\Omega)} \|\theta -\bar{\theta}\|_{L^4(\Omega)}\|\nabla \bar{\theta}\|_{\mathbf{L}^4(\Omega)}d\tau \non\\
&\leq & \epsilon \int_0^t \| \ub- \bar{\mathbf{u}} \|_{\mathbf{Z}}^2 d\tau + \epsilon \int_0^t \|\nabla (\theta-\bar{\theta})\|_{\mathbf{L}^2(\Omega)}^2 d\tau\non\\
&& + \frac{C}{\epsilon^7}\int_0^t \|\nabla \bar{\theta}\|_{\mathbf{L}^4(\Omega)}^8\|\theta-\bar{\theta}\|_{L^2(\Omega)}^2 d\tau.\label{I8}
\end{eqnarray}

Collecting the above estimates for $I_1,...,I_8$ and taking the constant $\epsilon$ to be sufficiently small (cf. assumptions (A1)--(A3)), we deduce from \eqref{adiff1} and \eqref{adiff2} that
\begin{eqnarray}
&& \|\mathbf{u}_f(t)-\bar{\mathbf{u}}_f(t)\|_{\mathbf{L}^2(\Omega_f)}^2+
\varpi\|\mathbf{u}_m(t)-\bar{\mathbf{u}}_m(t)\|_{\mathbf{L}^2(\Omega_m)}^2
+\|\theta(t)-\bar{\theta}(t)\|_{L^2(\Omega)}^2
 \non\\
 && +2\int_0^t\underline{\nu}\|\mathbb{D}(\ub_f-\bar{\ub}_f)\|_{\mathbf{L}^2(\Omega_f)}^2 d\tau
 +\int_0^t \left( \underline{\nu}\bar{\kappa}^{-1}(\ub_m-\bar{\ub}_m), (\ub_m-\bar{\ub}_m)\right)_m d\tau\non\\
 && +\sum_{j=1}^{d-1} \int_0^t\!\int_{\Gamma_{i}}\alpha \underline{\nu}\bar{\kappa}^{-\frac12} |(\ub_f-\bar{\ub}_f)\cdot\btau_j|^2 dS d\tau
  +\int_0^t \underline{\lambda}\|\nabla(\theta-\bar{\theta})\|_{\mathbf{L}^2(\Omega)}^2 d\tau
 \non\\
 &\leq & C\int_0^t h(\tau)\left( \|\ub_f-\bar{\ub}_f\|_{\mathbf{L}^2(\Omega_f)}^2+ \|\theta-\bar{\theta}\|_{L^2(\Omega)}^2\right) d\tau,\non
\end{eqnarray}
where
$$
h(t)= \|\bar{\ub}_f(t)\|_{\mathbf{W}^{1,6}(\Omega_f)}^4 + \|\bar{\mathbf{u}}_m(t)\|_{\mathbf{L}^6(\Omega_m)}^4 + \|\nabla \bar{\theta}(t)\|_{\mathbf{L}^4(\Omega)}^8+1.
$$
By the additional regularity assumptions on $(\bar{\ub}_f,\bar{\ub}_m,\bar{\theta})$, i.e.,
$$\bar{\ub}_f\in L^4(0,T;\mathbf{W}^{1,6}(\Omega_f)),\quad  \bar{\ub}_m\in L^4(0,T;\mathbf{L}^6(\Omega_m)),\quad \bar{\theta}\in L^8(0,T;W^{1,4}(\Omega)),$$
we see that  $h(t)\in L^1(0,T)$. This enables us to apply Gronwall's lemma and Lemma \ref{equinorml} to conclude that
$$
\|\mathbf{u}_f(t)-\bar{\mathbf{u}}_f(t)\|_{\mathbf{L}^2(\Omega_f)}^2+
\|\theta(t)-\bar{\theta}(t)\|_{L^2(\Omega)}^2 +\int_0^t \|\mathbf{u}(\tau)-\bar{\mathbf{u}}(\tau)\|_{\mathbf{Z}}^2 d\tau =0,
$$
for a.a. $t\in [0,T]$. As a consequence, we obtain the weak-strong uniqueness result for problem \eqref{uf1}--\eqref{IBCi2} with $\varpi\geq 0$.

The proof of Theorem \ref{thmuniq} is complete. \hfill $\square$

\section*{Acknowledgments}
The first author was partially supported by NNSFC 11871159 and Guangdong Provincial Key Laboratory for Computational Science and Material Design
2019B030301001. The second author was partially supported by NNSFC 12071084 and the Shanghai Center for Mathematical Sciences.


\end{document}